\documentclass[a4paper]{svproc}

\usepackage{url}

\usepackage{amsmath,amscd}
\usepackage{bbm}

\usepackage{dsfont}
\usepackage{bm}
\usepackage{amsthm,array}
\usepackage{mathtools}
\usepackage{amsfonts,latexsym}
\usepackage{graphicx,wrapfig}
\usepackage{subfig}
\usepackage{times}
\usepackage{psfrag,epsfig}
\usepackage{verbatim}
\usepackage{tabularx}
\usepackage[hidelinks]{hyperref}
\usepackage{color}
\usepackage{soul}
\usepackage{indentfirst}
\usepackage{inputenc}
\usepackage[capitalise]{cleveref}
\usepackage{courier}
\usepackage{authblk}
\usepackage{mathtools}
\usepackage{mathtools}
\usepackage{cite}
\usepackage{amssymb}
\usepackage{graphicx}
\usepackage{caption}
\usepackage{dblfloatfix}
\usepackage{lineno}
\usepackage{multirow}
\usepackage{algorithm,algpseudocode}
\usepackage{ltablex}
\usepackage{relsize}

\usepackage{enumitem}
\usepackage{xtab}
\usepackage{xfrac}
\usepackage{float}
\usepackage{url}

\xentrystretch{-0.07}

\makeatletter
\def\mathcenterto#1#2{\mathclap{\phantom{#1}\mathclap{#2}}\phantom{#1}}
\let\old@widetilde\widetilde
\def\widetildeto#1#2{\mathcenterto{#2}{\old@widetilde{\mathcenterto{#1}{#2\,}}}}
\let\old@widehat\widehat
\def\widehatto#1#2{\mathcenterto{#2}{\old@widehat{\mathcenterto{#1}{#2\,}}}}
\let\old@overline\overline
\def\overlineto#1#2{\mathcenterto{#2}{\old@overline{\mathcenterto{#1}{#2\,}}}}
\makeatother

\DeclareMathOperator*{\trace}{\mathrm{trace}}

\def\0{\boldsymbol{0}}

\def\b{\boldsymbol{b}}

\def\I{\mathbf{I}}

\def\transpose{\top} 

\def\R{\mathbb{R}}

\def\b0{\mathbf{0}}

\def\LL{\mathcal{L}}
\def\XX{\mathcal{X}}

\def\se3{\mathfrak{se}(3)}

\def\nR{{\widetilde{R}}}
\def\nt{\tilde{t}}
\def\nM{\widetildeto{H}{M}}

\def\lgamma{{\overline{\gamma}}}

\def\diag{\mathrm{diag}}

\def\grad{\mathrm{grad\,}}

\def\nH{{\widetilde{\mathrm{\Omega}}}}

\def\nGamma{\widetilde{\mathrm{\Gamma}}}
\def\rGamma{\mathrm{\Gamma}}
\def\nOmega{{\widetilde{\mathrm{\Omega}}}}

\def\gpm{\mathsf{GPM-PGO}}
\def\nagm{\mathsf{NAG-PGO}}
\def\gpmo{\mathsf{GPM-PGO^*}}
\def\agpm{\mathsf{AGPM-PGO}}
\def\agpmo{\mathsf{AGPM-PGO^*}}
\def\nagmo{\mathsf{NAG-PGO^*}}
\def\sesync{\mathsf{SE-Sync}}

\newcommand{\Xk}[1]{X^{(#1)}}
\newcommand{\Yk}[1]{Y^{(#1)}}
\newcommand{\Rk}[1]{R^{(#1)}}
\newcommand{\tk}[1]{t^{(#1)}}
\newcommand{\thk}[1]{\theta^{(#1)}}

\def\QQ{\mathcal{Q}}

\def\EE{\mathcal{E}}
\def\VV{\mathcal{V}}
\def\aEE{\overrightarrow{\EE}}
\def\aGG{\overrightarrow{G}}

\long\def\answer#1{}

\long\def\comment#1{}

\newcommand{\ten}[2]{{#1}\times 10^{#2}}

\theoremstyle{definition}

\crefname{prop}{Proposition}{Propositions}

\theoremstyle{remark}
\newtheorem*{remark*}{{Remark}}

\newcolumntype{P}[1]{>{\centering\arraybackslash}p{#1}}
\newcolumntype{M}[1]{>{\centering\arraybackslash}m{#1}}

\newcommand{\RNum}[1]{\uppercase\expandafter{\romannumeral #1\relax}}

\hypersetup{
	bookmarks=true,         
	unicode=false,          
	pdftoolbar=true,        
	pdfmenubar=true,        
	pdffitwindow=false,     
	pdfstartview={FitH},    
	pdftitle={My title},    
	pdfauthor={Author},     
	pdfsubject={Subject},   
	pdfcreator={Creator},   
	pdfproducer={Producer}, 
	pdfkeywords={keyword1} {key2} {key3}, 
	pdfnewwindow=true,      
	colorlinks=true,       
	linkcolor=blue,          
	citecolor=blue,        
	filecolor=magenta,      
	urlcolor=cyan           
}

\begin{document}
\mainmatter              
\title{Generalized Proximal Methods for Pose Graph Optimization}
\titlerunning{ Pose Graph Optimization}  
%
\author{Taosha Fan\inst{1} \and Todd Murphey\inst{1}}
\authorrunning{T. Fan and T. Murphey} 
%
\tocauthor{Taosha Fan and Todd Murphey}
\institute{Northwestern University, 2145 Sheridan Rd, Evanston, IL 60208,\\
	\email{taosha.fan@u.northwestern.edu}, \email{t-murphey@northwestern.edu}}

\maketitle              

\begin{abstract}
	\vspace{-1em}
	\sloppy	In this paper, we generalize proximal methods that were originally designed for convex optimization on normed vector space to non-convex pose graph optimization (PGO) on special Euclidean groups, and show that our proposed generalized proximal methods for PGO converge to first-order critical points. Furthermore, we propose methods that significantly accelerate the rates of convergence almost without loss of any theoretical guarantees.  In addition, our proposed methods can be easily distributed and parallelized with no compromise of efficiency. The efficacy of this work is validated through implementation on simultaneous localization and mapping (SLAM) and distributed 3D sensor network localization, which indicate that our proposed methods are a lot faster than existing techniques to converge to sufficient accuracy for practical use.
	\vspace{-0.5em}
\end{abstract}

\section{Introduction}
{Pose graph optimization} (PGO) estimates a number of unknown poses from noisy relative measurements, in which we associate each pose with a vertex and each measurement with an edge of a graph. PGO has important applications in a number of areas, for example, simultaneous localization and mapping (SLAM) in robotics \cite{cadena2016past}, structural analysis of biological macromolecules in cryo-electron microscopy \cite{singer2011three}, sensor network localization in distributed sensing \cite{tron2009distributed}, etc.

In the last twenty years, a number of PGO methods have been developed, which are either first-order optimization methods \cite{olson2006fast,grisetti2009nonlinear,tron2009distributed} or second-order optimization methods \cite{kaess2012isam2,rosen2014rise,kuemmerle11icra,rosen2016se}. In general, first-order PGO methods typically converge slowly when close to critical points, and thus, second-order PGO methods are preferable in most applications. In spite of this, second-order PGO methods have to continuously solve linear systems to evaluate descent directions, which is difficult to distribute and parallelize, and can be time-consuming for large-scale optimization problems \cite{kaess2012isam2,rosen2014rise,kuemmerle11icra,rosen2016se,fan2019iros}.  

In optimization and applied mathematics, there are a number of algorithms to accelerate the rates of convergence of  first-order optimization methods \cite{nesterov1983method,nesterov2013introductory,ghadimi2016accelerated,jin2018accelerated,li2015accelerated}. Nevertheless, most of existing accelerated first-order optimization methods \cite{nesterov1983method,nesterov2013introductory,ghadimi2016accelerated,jin2018accelerated,li2015accelerated} rely on proximal methods \cite{parikh2014proximal} and need a proximal operator that is also an upper bound of the objective function, and such a proximal operator, though exists, it is usually unclear for PGO. In addition, it is common in first-order PGO methods to formulate PGO as optimization on special Euclidean groups, and update pose estimates using Riemannian instead of Euclidean gradients \cite{olson2006fast,grisetti2009nonlinear,tron2009distributed}, whereas in general accelerated first-order optimization methods \cite{nesterov1983method,nesterov2013introductory,ghadimi2016accelerated,jin2018accelerated,li2015accelerated} only apply to optimization on normed vector space and are inapplicable for optimization using Riemannian gradients. As a result, it is in strictly limited to accelerate existing first-order PGO methods \cite{olson2006fast,grisetti2009nonlinear,tron2009distributed} with \cite{nesterov1983method,nesterov2013introductory,ghadimi2016accelerated,jin2018accelerated,li2015accelerated}.

In this paper, we generalize proximal methods \cite{parikh2014proximal} that were originally designed for convex optimization on normed vector space to non-convex PGO on special Euclidean groups, and show that our proposed methods converge to first-order critical points. Different from existing first-order PGO methods \cite{olson2006fast,grisetti2009nonlinear,tron2009distributed}, our proposed methods do not rely on Riemannian gradients to update pose estimates and there is no need to perform line search to guarantee convergence. Instead, our proposed methods update pose estimates by solving optimization sub-problems in closed form. Furthermore, we present methods that significantly accelerate the rates of convergence using \cite{nesterov1983method,nesterov2013introductory} with no loss of theoretical guarantees. To  our knowledge, neither proximal methods nor accelerated first-order methods for PGO have been presented before. In addition, our proposed methods can be easily distributed and parallelized without compromise of efficiency. In spite of being first-order PGO methods, our proposed methods are empirically several times faster than second-order PGO methods to converge to modest accuracy that is sufficient for practical use. In cases when higher accuracy is required, our proposed methods can  be combined with second-order PGO methods \cite{kaess2012isam2,rosen2014rise,kuemmerle11icra,rosen2016se,fan2019iros} to improve the overall performance.

The rest of this paper is organized as follows. \cref{section::notation} introduces notations that are used throughout this paper. \cref{section::gpm} reformulates proximal methods in a more general way that is used in this paper to solve PGO. \cref{section::formulation} formulates and simplifies PGO. \cref{section::relax} proposes a generalized proximal operator that is also an upper bound of PGO, which is fundamental to our proposed methods. \cref{section::acc,section::fast} present unaccelerated and accelerated generalized proximal methods for PGO, respectively, which is the major contribution of this paper. \cref{section::experiments} implements our proposed methods on SLAM and distributed 3D sensor network localization, and makes comparisons with $\sesync$ \cite{rosen2016se}. The conclusions are made in \cref{section::conclusion}

\vspace{-0.5em}
\section{Notation}\label{section::notation}
$\R$ denotes the sets of real numbers; $\R^{m\times n}$ and $\R^n$ denote the sets of $m\times n$ matrices and $n\times 1$ vectors, respectively; and $SO(d)$ and $SE(d)$ denote the sets of special orthogonal groups and special Euclidean groups, respectively. For a matrix $X\in \R^{m\times n}$, the notation $[X]_{ij}$ denotes the $(i,\,j)$-th entry or $(i,\,j)$-th block of $X$. The notation $\|\cdot\|$ denotes the Frobenius norm of matrices and vectors. For symmetric matrices $Y,\, Z\in \R^{n\times n}$, $Y\succeq Z$ (or $Z\preceq Y$) and $Y\succ Z$ (or $Z\prec Y$) mean that $Y-Z$ is positive semidefinite and positive definite, respectively. If $F:\R^{m\times n}\rightarrow\R $ is a function, $\mathcal{M}\subset \R^{m\times n}$ is a manifold and $X\in \mathcal{M}$, the notation $\nabla F(X)$ and $\mathrm{grad}\, F(X)$ denote the Euclidean and Riemannian gradients, respectively. 

\section{Generalized Proximal Methods}\label{section::gpm}
For an optimization problem 
\begin{equation}
\nonumber
\min_{X\in\XX} F(X),
\end{equation}
in which $\mathcal{X}$ is a closed set and $F:\XX\rightarrow \R$ is a function with Lipschitz smooth gradient $\nabla F(X)$ for a scalar $L>0$ such that
\begin{equation}\label{eq::L}
F(Y)\leq F(X) + \nabla F(X)^\transpose (Y-X)+\frac{L}{2}\|Y-X\|^2,
\end{equation}
then the proximal operator of the first-order approximation at $X^{(k)}\in \XX$ is defined to be \cite{parikh2014proximal}
\begin{equation}\label{eq::prox}
X^{(k+1)}=\arg\min_{X\in \XX} F(X^{(k)}) + \nabla F(X^{(k)})^\transpose\big(X-X^{(k)}\big)+\frac{L}{2}\|X-X^{(k)}\|^2,
\end{equation}
from which it can be concluded that $F(X^{(k+1)})\leq F(X^{(k)})$. An optimization algorithm using \cref{eq::prox} to generate iterates $\{X^{(k)}\}$ is called the {proximal method}. Though originally designed for convex optimization \cite{parikh2014proximal,nesterov1983method,nesterov2013introductory}, proximal methods have been used to solve non-convex optimization problems and get quite good results \cite{ghadimi2016accelerated,li2015accelerated,jin2018accelerated}. 

From \cref{eq::prox}, a prerequisite of proximal methods is that there exists a positive scalar $L$ w.r.t. which $F(X)$ is Lipschitz smooth. In most cases, $L$ is unknown, and finding such a scalar $L$ can be time-consuming. Instead, if there exists a positive definite matrix $\mathrm{\Omega} $ such that
\begin{equation}\label{eq::omega}
F(Y)\leq F(X) + \nabla F(X)^\transpose (Y-X)+\frac{1}{2}(Y-X)^\transpose \mathrm{\Omega}(Y-X),
\end{equation}
we obtain an first-order approximation that is also an upper bound of $F(X)$ as
\begin{multline}\label{eq::gpro}
X^{(k+1)}=\arg\min_{X\in \XX} F(X^{(k)}) + \nabla F(X^{(k)})^\transpose\big(X-X^{(k)}\big)+\\
\frac{1}{2}(X-X^{(k)})^\transpose\mathrm{\Omega}(X-X^{(k)}).
\end{multline}
We term \cref{eq::gpro} as the generalized proximal operator and an optimization algorithm using the equation above to generate iterates $\{X^{(k)}\}$ as the {generalized proximal method}. For a number of optimization problems, finding a matrix $\mathrm{\Omega}$ satisfying \cref{eq::omega} is much easier than finding a scalar $L$ satisfying \cref{eq::L}. Even though it is possible to determine a scalar $L>0$ satisfying \cref{eq::L} as the greatest eigenvalue of $\mathrm{\Omega}$, it is still expected that \cref{eq::gpro} results in a better approximation and a tighter upper bound than \cref{eq::prox}.

In the following sections, we will propose generalized proximal methods using \cref{eq::gpro} to solve PGO. 

\vspace{-0.5em}
\section{Problem Formulation}\label{section::formulation}
In this section, we review PGO that can be formulated as a least-square optimization problem and be simplified to a compact quadratic form. It should be noted that both the least-square and quadratic formulations of PGO have been well addressed by Rosen \textit{et. al} in \cite{rosen2016se}, and due to space limitations, we only present the main results and interested readers can refer to \cite{rosen2016se} for a detailed introduction.

PGO estimates $n$ unknown poses $g_i\triangleq(R_i,\,t_i)\in SE(d)$ with $m$ noisy measurements of relative poses $g_{ij}\triangleq g_i^{-1}g_j\triangleq (R_{ij},\,t_{ij})\in SE(d)$. In PGO, the $n$ poses $g_i$ and  $m$ relative measurements $g_{ij}$ are described through a directed graph $\aGG\triangleq(\VV, \, \aEE)$ in which $\VV\triangleq\{1,\,\cdots,\, n\}$ and each index $i$ is associated with $g_i$, and $(i,\,j)\in \aEE\subset\VV\times\VV$ if and only if $g_{ij}$ exists. 
If we ignore the orientation of edges in $\EE$, an undirected graph $G\triangleq(\VV,\,\EE)$ is obtained. In the rest of this paper, it is assumed that $\aGG$ is weakly connected and $G$ is connected. Following \cite{rosen2016se}, we also assume that the $m$ measurements $(R_{ij},\,t_{ij})$ are random variables:
\begin{subequations}
	\begin{equation}
	\nt_{ij}=\underline{t}_{ij}+t_{ij}^\epsilon\quad\quad\quad t_{ij}^\epsilon\sim \mathcal{N}(\0,\, \tau_{ij}^{-1}\I),
	\end{equation}
	\begin{equation}\label{eq::tob}
	\,\nR_{ij}=\underline{R}_{ij}R_{ij}^\epsilon\quad\quad\;\;\;\; R_{ij}^\epsilon\sim \mathcal{L}(\I,\, \kappa_{ij}),\hphantom{dd}
	\end{equation}
\end{subequations}
in which $\underline{g}_{ij}\triangleq(\underline{R}_{ij},\,\underline{t}_{ij})\in SE(d)$ is the true (latent) value of $g_{ij}$, and $\mathcal{N}(\mu_t,\,\mathrm{\Sigma}_t)$ denotes the normal distribution with mean $\mu_t \in \R^d$ and covariance $0\preceq \mathrm{\Sigma}_t\in \R^{d\times d}$, and $\LL(\mu_{R},\,\kappa_R)$ denotes the isotropic Langevin distribution with mode $\mu_{R}\in SO(d)$ and concentration parameter $\kappa_R \geq 0$.

\sloppy From the perspective of maximum likelihood estimation\cite{rosen2016se}, PGO can be formulated as a least square optimization problem on $SE(d)^n$
\begin{equation}\label{eq::obj}
\min_{\substack{R_i\in SO(d),\,t_i\in R^d,\\ i=1,\,\cdots,\,n}}\sum_{(i,\,j)\in\aEE}\frac{1}{2}\left(\kappa_{ij}\cdot\|R_i\nR_{ij}-R_{j}\|^2+
\tau_{ij}\cdot\|R_{i}\nt_{ij}+t_i-t_j\|^2\right),
\end{equation}
in which $R_i\in SO(d)$ and $t_i\in \R^d$. A straightforward derivation further simplifies \cref{eq::obj} to 
\begin{equation}\label{eq::obj_quad}
\min_{X\in \R^{d\times n}\times SO(d)^n} F(X)\triangleq\dfrac{1}{2}\trace(X \nM  X^\transpose)
\end{equation}
in which $F(X)$ is a quadratic function  and $X\triangleq\begin{bmatrix}
t_1 & \cdots & t_n & R_1 &\cdots & R_n
\end{bmatrix}\in \R^{d\times n}\times SO(d)^n\subset\R^{d\times(d+1)n}$. For $F(X)$ of \cref{eq::obj_quad}, $\nM\in \R^{(d+1)n\times(d+1)n}$ is a positive-semidefinite matrix
\begin{equation}\label{eq::M}
\nM\triangleq \begin{bmatrix}
L(W^\tau) & \widetilde{V}\\
\widetilde{V}^\transpose & L(\widetilde{G}^\rho)+\widetilde{\mathrm{\Sigma}}
\end{bmatrix}\in \R^{(d+1)n\times (d+1)n},
\end{equation}
in which $L(W^\tau)\in \R^{n\times n}$, $L(\widetilde{G}^\rho)\in \R^{dn\times{dn}}$, $\widetilde{V}\in \R^{n\times dn}$ and $\widetilde{\mathrm{\Sigma}}=\diag\{\widetilde{\mathrm{\Sigma}}_i,\,\cdots,\,\widetilde{\mathrm{\Sigma}}_n\}\in \R^{dn\times dn}$ are sparse matrices defined as Eqs. (13) to (16) in \cite[Section 4]{rosen2016se}.

In the next section, we will propose a generalized proximal operator of \cref{eq::obj,eq::obj_quad} whose minimization is  $n$ independent optimization problems on $SE(d)$, and thus can be efficiently solved, which is fundamental to our proposed methods for PGO.

\vspace{-0.5em}
\section{The Generalized Proximal Operator for PGO
}\label{section::relax}
In this section, we propose  an upper bound of PGO, and show that the resulting upper bound is a generalized proximal operator of the first-order approximation of PGO. 

For any matrices $A$ and $B$ of the same size, it is known that
\vspace{-0.1em}
\begin{equation}\label{eq::diff}
\frac{1}{2}\|A-B\|^2 = \min_{P\in \R^{m\times n}} \|A-P\|^2+\|B-P\|^2,
\vspace{-0.75em}
\end{equation}
the unique optimal solution to which is $P=\dfrac{1}{2}A+\dfrac{1}{2}B$. As a result of \cref{eq::diff}, if we introduce $m$ pairs of extra variables $P_{ij}\in \R^{d\times d}$ and $p_{ij}\in \R^d$ for each $(i,\,j)\in \aEE$ to \cref{eq::obj}, an upper bound of PGO is obtained as 
\vspace{-0.5em}
\begin{multline}\label{eq::objn}
\min_{\substack{R_i\in SO(d),\,t_i\in \R^d,\\ i=1,\,\cdots,\,n}}\,\sum\limits_{(i,\,j)\in\aEE}\left(\kappa_{ij}\cdot\|R_i\nR_{ij}-P_{ij}\|^2+\tau_{ij}\cdot\|R_{i}\nt_{ij}+t_i-p_{ij}\|^2+\vphantom{\kappa_{ij}\cdot\|R_j-P_{ij}\|^2+\tau_{ij}\cdot\|t_j-p_{ij}\|^2}\right.\\[-0.5em]
\left.\vphantom{\kappa_{ij}\cdot\|R_i\nR_{ij}-P_{ij}\|^2+\tau_{ij}\cdot\|R_{i}\nt_{ij}+t_i-p_{ij}\|^2}\kappa_{ij}\cdot\|R_j-P_{ij}\|^2+\tau_{ij}\cdot\|t_j-p_{ij}\|^2\right).
\end{multline}
\vspace{-0.5em}
If $P_{ij}$ and $p_{ij}$ are chosen as 
\begin{equation}\label{eq::P}
P_{ij}=\frac{1}{2}\Rk{k}_{i}\nR_{ij} + \frac{1}{2}\Rk{k}_{j},\quad\quad p_{ij}=\frac{1}{2}\Rk{k}_{i}\nt_{ij}+\frac{1}{2}\tk{k}_i + \frac{1}{2}\tk{k}_j,
\end{equation}
with $(\Rk{k}_{i},\,\tk{k}_i)\in SE(d)$ and $i=1,\,\cdots,\,n$, then \cref{eq::obj} and \cref{eq::objn} attain the same objective value at $\Rk{k}_{i}$ and $\tk{k}_i$. As a matter of fact, \cref{eq::objn} results in a generalized proximal operator of the first-order approximation of PGO as stated in \cref{theorem::first-order}.

\begin{theorem}\label{theorem::first-order}
	Let $P_{ij}$ and $p_{ij}$ are chosen as \cref{eq::P} with $(\Rk{k}_{i},\,\tk{k}_i)\in SE(d)$ and $i=1,\,\cdots,\,n$. Then, there exists a constant matrix $0\preceq\nH\in \R^{(d+1)n\times(d+1)n}$ such that \cref{eq::objn} is equivalent to
	\vspace{-0.25em}
	\begin{multline}\label{eq::obj_first}
	\min_{X\in \R^{d\times n}\times SO(d)^n} F(\Xk{k})+\trace\left(({X-\Xk{k}})^\transpose{\nabla F(\Xk{k})}\right)+\\[-0.15em]
	\frac{1}{2}\trace\left(\big(X-\Xk{k}\big) \nH\big(X-\Xk{k}\big)^\transpose\right)
	\end{multline}
	\vspace{-0.25em}
	in which $\Xk{k}=\begin{bmatrix}
	\tk{k}_1 & \cdots & \tk{k}_n & \Rk{k}_1 &\cdots & \Rk{k}_n
	\end{bmatrix}\in \R^{d\times n}\times SO(d)^n$, $F(\Xk{k})=\dfrac{1}{2}\trace(\Xk{k} \nM {\Xk{k}}^\transpose)$ and $\nabla F(\Xk{k})= \Xk{k}\nM$.
\end{theorem}
\begin{proof}
	See \cite[Appendix B.1]{fan2019gpm}.
\end{proof}

It should be noted that \cref{eq::obj_first} is an upper bound of PGO as well as \cref{eq::obj,eq::obj_quad}, in which $\nM$ and $\nOmega$ are closely related as follows.

\begin{theorem}\label{theorem::bound}
	Let $\nM$ and $\nOmega$ be defined as \cref{eq::obj_quad,eq::obj_first}, respectively. Then,
	\vspace{-0.1em}
	\begin{enumerate}[label={(\alph*)}]
		\item  $\nOmega\succeq \nM$;\label{item::bound_1}\vspace{0.2em}
		\item  for any $c\in \R$, $c\cdot\I \succeq \nOmega$ if $\dfrac{c}{2}\cdot \I \succeq \nM$.\label{item::bound_2}
	\end{enumerate}
\end{theorem} 
\begin{proof}
	See \cite[Appendix B.2]{fan2019gpm}.
\end{proof}

As a result of \cref{theorem::first-order,theorem::bound}, \cref{eq::obj_first} is a generalized proximal operator and an upper bound of PGO, which suggests the possibility of generalized proximal methods to solve PGO. It should be noted that only the Euclidean gradient $\nabla F(\Xk{k})$ is involved in \cref{eq::obj_first}, and as a result, we might also accelerate generalized proximal methods using \cite{nesterov1983method,nesterov2013introductory}. In \cref{section::fast}, we will propose generalized proximal methods for PGO, and in \cref{section::acc}, we will further accelerate our proposed generalized proximal methods for PGO, which is the major contribution of this paper.

\vspace{-0.5em}
\section{Generalized Proximal Methods for PGO}\label{section::fast}
In this section, we propose two generalized proximal methods to solve PGO and show that our proposed methods for PGO converge to first-order critical points. 
\vspace{-1em}
\subsection{The $\gpm$ Method}
\vspace{-2em}
\setlength{\textfloatsep}{1em}
\begin{algorithm}[!htpb]
	\caption{The $\gpm$ Method}
	\label{algorithm::pm}
	\begin{algorithmic}[1]
		\State\textbf{Input}: An initial iterate $X^{(0)}=\begin{bmatrix}
		t_1^{(0)} & \cdots & t_n^{(0)} & R_1^{(0)} & \cdots & R_n^{(0)}
		\end{bmatrix}\in\R_d^n\times SO(d)^n$, and the maximum number of iterations $N$.
		\State\textbf{Output}: A sequence of iterates $\{X^{(k)}\}$.\vspace{0.2em} 
		\Function { $\gpm$}{$X^{(0)},\, N$}
		\For{$k=0\rightarrow N-1$}
		\vspace{0.25em}
		\State $\begin{bmatrix}
		\theta_1^{(k)} & \cdots &\theta_n^{(k)} 
		\end{bmatrix}\leftarrow X^{(k)}\mathrm{\Theta}$
		\vspace{0.25em}
		\For{$i=1\rightarrow n$}
		\vspace{0.2em}
		\State $R_i^{(k+1)}\leftarrow\arg\max\limits_{R_i\in SO(d) }\trace(R_i^\transpose\theta_i^{(k)})$
		\EndFor
		\vspace{0.25em}
		\State $R^{(k+1)}\leftarrow\begin{bmatrix}
		R_1^{(k+1)} & \cdots & R_n^{(k+1)}
		\end{bmatrix}$
		\vspace{0.25em}
		\State $t^{(k+1)}\leftarrow  R^{(k+1)}\mathrm{\Xi}+X^{(k)}\mathrm{\Psi}$
		\vspace{0.25em}
		\State $X^{(k+1)}\leftarrow\begin{bmatrix}
		t^{(k+1)} &  R^{(k+1)}
		\end{bmatrix}$
		\EndFor
		\State \textbf{return} $\{X^{(k)}\}$
		\EndFunction
	\end{algorithmic}
\end{algorithm}
\vspace{-0.5em}

According to \cref{theorem::bound}, it is known that $\nH\succeq \nM$, and thus, we obtain a series of upper bounds of PGO:
\begin{multline}\label{eq::obj_first2}
\min_{X\in \R^{d\times n}\times SO(d)^n} G(X|X^{(k)})\triangleq F(\Xk{k})+\trace\left(({X-\Xk{k}})^\transpose{\nabla F(\Xk{k})}\right)+\\\frac{1}{2}\trace\left(\big(X-\Xk{k}\big) \nGamma\big(X-\Xk{k}\big)^\transpose\right),
\end{multline}
in which $\nGamma = \nH+\alpha\cdot \I\succeq \nM$ and $\alpha\geq 0$. From \cref{eq::obj_first,eq::objn}, a straightforward mathematical manipulation indicates that $\nGamma$ in \cref{eq::obj_first2} takes the form as $\nGamma\triangleq\begin{bmatrix}
\rGamma^\tau & {\nGamma^\nu}{{}^\transpose}\\
\nGamma^\nu  & \nGamma^\rho\phantom{{}^\transpose}
\end{bmatrix}$ with $\rGamma^{\tau}=\diag\{\rGamma_1^\tau,\,\cdots,\,\rGamma_n^\tau\}\in \R^{n\times n}$, $\nGamma^\rho=\diag\{\nGamma_1^\rho,\,\cdots,\,\nGamma_n^\rho\}\in \R^{dn\times dn}$ and $\nGamma^\nu=\diag\{
\nGamma_1^\nu,\,  \cdots ,\, \nGamma_n^\nu \}\in \R^{dn\times n}$, in which 
\begin{subequations}
	\begin{equation}\label{eq::Htau}
	\rGamma_i^\tau = \alpha+\sum_{(i,\,j)\in\EE}2\cdot\tau_{ij} \in \R,
	\end{equation}
	\begin{equation}\label{eq::Hrho}
	\nGamma_i^\rho = \alpha\cdot \I +\sum\limits_{(i,\,j)\in \EE}2\cdot \kappa_{ij}\cdot\I+\sum\limits_{(i,\,j)\in \aEE}2\cdot\tau_{ij}\cdot\nt_{ij}\nt_{ij}^\transpose\in \R^{d\times d},
	\end{equation}
	\begin{equation}\label{eq::Hnu}
	\nGamma_i^\nu = \sum\limits_{(i,\,j)\in \aEE}2\cdot\tau_{ij}\cdot \nt_{ij}\in \R^{d}.
	\end{equation}
\end{subequations}
For simplicity and clarity, we rewrite $\nabla F(\Xk{k})$ in \cref{eq::obj_first} as  $\nabla F(\Xk{k})=\begin{bmatrix}
\lgamma_1^\tau & \cdots &\lgamma_n^\tau & \lgamma_1^\rho  & \cdots & \lgamma_n^\rho
\end{bmatrix}\in\R^{d \times (d+1)n}$, in which $\lgamma_i^\tau \in \R^d$ and $\lgamma_i^\rho \in \R^{d\times d}$ are Euclidean gradients w.r.t. $t_i$ and $R_i$, respectively.  Substituting \cref{eq::Hnu,eq::Hrho,eq::Htau} into \cref{eq::obj_first} and simplifying the resulting equation, we obtain
\begin{multline}\label{eq::obj_sub}
\min_{\substack{R_i\in SO(d),\,t_i\in \R^d,\\ i=1,\,\cdots,\,n}}\, F(\Xk{k})+\sum\limits_{i\in \VV}\Big(\frac{1}{2}\trace\left((R_i-\Rk{k}_{i})\nGamma_i^\rho (R_i-\Rk{k}_{i})^\transpose)\right)+\\\nGamma_i^\nu{}^\transpose (R_i-\Rk{k}_{i})^\transpose(t_i-\tk{k}_i) + \frac{1}{2} (t_i-\tk{k}_i)^\transpose \rGamma_i^\tau(t_i-\tk{k}_i)+\\
\trace\left(\lgamma_i^\rho{}^\transpose(R_i-\Rk{k}_i)\right)+\lgamma_i^\tau{}^\transpose(t_i -\tk{k}_i)\Big),
\end{multline}
which is equivalent to  $n$ independent optimization problems on $(R_i,\,t_i)\in SE(d)$:
\vspace{-0.5em}
\begin{multline}\label{eq::obj_sub1}
\min_{\substack{R_i\in SO(d),\,t_i\in \R^d}}\frac{1}{2}\trace\left((R_i-\Rk{k}_{i})\nGamma_i^\rho (R_i-\Rk{k}_{i})^\transpose)\right)+\\
\nGamma_i^\nu{}^\transpose (R_i-\Rk{k}_{i})^\transpose(t_i-\tk{k}_i) + \frac{1}{2}\rGamma_i^\tau (t_i-\tk{k}_i)^\transpose (t_i-\tk{k}_i)+\\
\trace\left(\lgamma_i^\rho{}^\transpose(R_i-\Rk{k}_i)\right)+
\lgamma_i^\tau{}^\transpose(t_i -\tk{k}_i).
\end{multline}
Furthermore, if $R_i \in SO(d)$ is given, $t_i\in \R^d$ can be recovered as 
\begin{equation}\label{eq::tsol}
t_i =\tk{k}_i- R_i\nGamma_i^\nu\rGamma_i^\tau{}^{-1}+ \big(\Rk{k}_i\nGamma_i^\nu-\lgamma_i^\tau\big)\rGamma_i^\tau{}^{-1}. 
\end{equation}
Substituting \cref{eq::tsol} into \cref{eq::obj_sub1} to cancel out $t_i$ and applying $R_iR_i^\transpose=\I$ to simplify the resulting equation, we obtain
\begin{equation}\label{eq::obj_R}
\Rk{k+1}_i=\arg \max_{R_i\in SO(d)} \trace(R_i^\transpose\thk{k}_i),
\vspace{-0.5em}
\end{equation}
in which 
\begin{equation}\label{eq::theta}
\thk{k}_i=\Rk{k}_{i}(\nGamma_i^\rho-\nGamma_i^\nu{\rGamma_i^\tau}{}^{-1}\nGamma_i^\nu{}^\transpose)+\lgamma_i^\tau\rGamma_i^\tau{}^{-1}\nGamma_i^\nu{}^{\transpose}-\lgamma_i^\rho\in \R^{d\times d}.
\end{equation}
From $\nabla F(\Xk{k})=\Xk{k} \nM$ and \cref{eq::M,eq::obj_first2}, we might matricize \cref{eq::theta} as 
\begin{equation}
\nonumber
\thk{k}  = \Xk{k} \mathrm{\Theta},
\end{equation}
in which $\thk{k} =\begin{bmatrix}
\thk{k}_1 & \cdots & \thk{k}_n
\end{bmatrix}\in \R^{d\times dn}$, and $\mathrm{\Theta}\in \R^{(d+1)n\times dn}$ is a sparse matrix
\begin{equation}\label{eq::theta2}
\mathrm{\Theta} = \begin{bmatrix}
L(W^\tau)\\
\widetilde{V}^\transpose
\end{bmatrix}\rGamma^\tau{}^{-1}\nGamma^\nu{}^\transpose+
\begin{bmatrix}
\0\\
\nGamma^\rho-\nGamma^\nu\rGamma^\tau{}^{-1}\nGamma^\nu{}^\transpose
\end{bmatrix}
-\begin{bmatrix}
\widetilde{V}\\
L(\widetilde{G}^\rho)+\widetilde{\mathrm{\Sigma}}
\end{bmatrix}.
\end{equation}
Similarly, as a result of \cref{eq::M,eq::tsol,eq::obj_first2}, we obtain
\begin{equation}\label{eq::tsol2}
\tk{k+1} = \Rk{k+1}\mathrm{\Xi}+\Xk{k}\mathrm{\Psi},
\end{equation}
in which $\tk{k+1}=\begin{bmatrix}
\tk{k+1}_1 & \cdots & \tk{k+1}_n
\end{bmatrix}\in \R^{d\times n}$, $\Rk{k+1}=\begin{bmatrix}
\Rk{k+1}_1 & \cdots & \Rk{k+1}_n
\end{bmatrix}\in SO(d)^n\subset\R^{d\times dn}$, and $\mathrm{\Xi}\in \R^{dn\times n}$ and $\mathrm{\Psi}\in \R^{(d+1)n\times n}$ are sparse matrices with $\mathrm{\Xi}= - \nGamma^\nu \rGamma^{\tau}{}^{-1} $
and $\mathrm{\Psi}= \begin{bmatrix}
\rGamma^\tau-L(W^\tau) \\
\nGamma^\nu -\widetilde{V}^\transpose
\end{bmatrix}\rGamma^\tau{}^{-1}$,
respectively. 

It is obvious that \cref{eq::obj_sub1} is simplified to \cref{eq::obj_R}. From \cite{umeyama1991least}, if $\thk{k}_i\in \R^{d\times d}$ admits a singular value decomposition $\thk{k}_i=U_i\mathrm{\Sigma}_i V_i^\transpose$ in which $U_i$ and $V_i\in O(d)$ are orthogonal (but not necessarily special orthogonal) matrices, and $\mathrm{\Sigma}_i=\diag\{\sigma_1,\,\sigma_2,\,\cdots,\,\sigma_d\}\in \R^{d\times d}$ is a diagonal matrix, and $\sigma_1\geq\sigma_2\geq\cdots\geq \sigma_d\geq 0$ are singular values of $\theta_i$, then the optimal solution to \cref{eq::obj_R} is
\vspace{-0.25em}
\begin{equation}\label{eq::Rsol}
R_i=\begin{cases}
U_i\mathrm{\Sigma}^+ V_i^\transpose, & \det(U_iV_i^\transpose) > 0,\\
U_i\mathrm{\Sigma}^- V_i^\transpose, & \det(U_iV_i^\transpose) < 0,
\end{cases}
\vspace{-0.25em}
\end{equation}
in which $\mathrm{\Sigma}^+=\diag\{1,\,1,\,\cdots,\,1\}$ and $\mathrm{\Sigma}^-=\diag\{1,\,1,\,\cdots,\,-1\}$. If $d = 2$, the equation above is equivalent to the polar decomposition of $2\times 2$ matrices, and if $d=3$, there are fast algorithms for singular value decomposition of $3\times 3$ matrices \cite{mcadams2011computing}. In both cases of $d=2$ and $d=3$, \cref{eq::obj_R} can be efficiently solved. As long as $R_i\in SO(d)$ is known, we can further recover $t_i\in \R^d$ using \cref{eq::tsol2} so that a solution $(t_i,\,R_i)\in \R^d\times SO(d)$ to \cref{eq::obj_sub1} is obtained with which \cref{eq::obj_sub} is also solved.

Therefore, \cref{eq::obj_sub} only involves $n$ singular value decomposition to solve \cref{eq::obj_R} on $R_i\in SO(d)$, and a matrix-vector multiplication to retrieve $t_i\in \R^d$ using \cref{eq::tsol2}, which suggests the $\gpm$ method (\cref{algorithm::pm}). 
\vspace{-1em}
\subsection{The $\gpmo$ Method}

It is according to \cref{eq::obj_quad} that PGO can be reformulated as
\vspace{-0.5em} 
\begin{multline}
\min_{t\in \R^{d\times n},\, R\in SO(d)^n}\frac{1}{2} \trace\big(t L(W^\tau) t^\transpose\big) +\trace\big(\tilde{V} R^\transpose t \big)+\\
\frac{1}{2}\trace(RL(\widetilde{G}^\rho)R^\transpose) + \frac{1}{2}\trace(R\widetilde{\mathrm{\Sigma}}^\rho R^\transpose),
\end{multline}
in which $t=\begin{bmatrix}
t_1 & \cdots & t_n
\end{bmatrix}\in \R^{d\times n}$ and $R=\begin{bmatrix}
R_1 & \cdots & R_n
\end{bmatrix}\in SO(d)^n$. Following a similar procedure in \cite{rosen2016se}, if the rotation $R=\Rk{k+1}$ in the equation above is given, we can optimally recover the corresponding translation $t=\tk{k+1}$ as $\tk{k+1} = -\Rk{k+1} \widetilde{V}^\transpose L(W^\tau)^\dagger$,
which is further simplified to 
\vspace{-0.25em}
\begin{equation}\label{eq::tsolf}
\tk{k+1} = -\Rk{k+1} \widetilde{T}^\transpose \mathrm{\Omega} A^\transpose (A\mathrm{\Omega}A^\transpose)^\dagger.
\vspace{-0.25em}
\end{equation}
In \cref{eq::tsolf}, $A\in \R^{n\times m}$, $\mathrm{\Omega}\in \R^{m\times m}$ and $\widetilde{T}\in \R^{m\times dn}$ are sparse matrices that are defined as Eqs.(7), (22) and (23) in \cite[Sections 3 and 4]{rosen2016se}, respectively.

As a result, instead of computing each $t_i^{(k+1)}\in \R^d$ sub-optimally  using \cref{eq::tsol,eq::tsol2}, we might use \cref{eq::tsolf} to optimally recover $t^{(k+1)}\in \R^{d\times n}$ w.r.t. $R^{(k+1)}\in SO(d)^n$ as a whole. Furthermore, if $t^{(k+1)}$ is recovered by \cref{eq::tsolf}, we have $\lgamma_i^\tau =\0$ for all $i=1,\,\cdots,\,n$ in \cref{eq::obj_sub1}, and thus, only the Euclidean gradient $\lgamma^\rho=\begin{bmatrix}
\lgamma_1^\rho & \cdots & \lgamma_n^\rho
\end{bmatrix}=\nabla_R F(X^{(k)})$ w.r.t. $R^{(k)}$ needs to be computed, and then $\theta_i^{(k)}$ is simplified to
\vspace{-0.5em}
\begin{equation}
\nonumber
\theta_i^{(k)}=R_{i}^{(k)}(\nGamma_i^\rho-\nGamma_i^\nu{\rGamma_i^\tau}{}^{-1}\nGamma_i^\nu{}^\transpose)-\lgamma_i^\rho.
\vspace{-0.25em}
\end{equation}
Following a similar procedure to \cref{eq::theta2}, we obtain
\vspace{-0.25em}
\begin{equation}
\nonumber
\theta^{(k)}=\Xk{k}\mathrm{\Phi},
\vspace{-0.25em}
\end{equation}
in which $\theta^{(k)} =\begin{bmatrix}
\theta_1^{(k)} & \cdots & \theta_n^{(k)}
\end{bmatrix}\in \R^{d\times dn}$, and $\mathrm{\Phi}\in \R^{(d+1)n\times dn}$ is a sparse matrix 
\vspace{-0.25em}
\begin{equation}\label{eq::phi}
\mathrm{\Phi}=\begin{bmatrix}
\0\\
\nGamma^\rho-\nGamma^\nu\rGamma^\tau{}^{-1}\nGamma^\nu{}^\transpose
\end{bmatrix}
-\begin{bmatrix}
\widetilde{V}\\
L(\widetilde{G}^\rho)+\widetilde{\mathrm{\Sigma}}
\end{bmatrix}.
\end{equation}
From \cref{eq::tsolf,eq::phi}, we obtain $\mathsf{GPO-PGO^*}$ (\cref{algorithm::pm*}),  which always recovers the translation $t$ optimally w.r.t. $R$ and thus is expected to outperform \cref{algorithm::pm}.

\begin{algorithm}[tb]
	\caption{The $\gpmo$ Method}
	\label{algorithm::pm*}
	\begin{algorithmic}[1]
		\State\textbf{Input}: An initial iterate $x^{(0)}=\begin{bmatrix}
		t^{(0)} & R^{(0)}
		\end{bmatrix}\in\R^{d\times n}\times SO(d)^n$ in which $t^{(0)}=-R^{(0)} \widetilde{T}^\transpose \mathrm{\Omega} A^\transpose (A\mathrm{\Omega}A^\transpose)^\dagger $, and the maximum number of iterations $N$.
		\State\textbf{Output}: A sequence of iterates $\{X^{(k)}\}$.\vspace{0.2em} 
		\Function { $\gpmo$}{$X^{(0)},\, N$}
		\vspace{0.1em}
		\For{$k=0\rightarrow N-1$}
		\vspace{0.15em}
		\State $\begin{bmatrix}
		\theta_1^{(k)} & \cdots &\theta_n^{(k)} 
		\end{bmatrix}\leftarrow X^{(k)}\mathrm{\Phi}$
		\vspace{0.15em}
		\For{$i=1\rightarrow n$}
		\vspace{0.1em}
		\State $R_i^{(k+1)}\leftarrow\arg\max\limits_{R_i\in SO(d) }\trace(R_i^\transpose\theta_i^{(k)})$
		\EndFor
		\vspace{0.15em}
		\State $R^{(k+1)}\leftarrow\begin{bmatrix}
		R_1^{(k+1)} & \cdots & R_n^{(k+1)}
		\end{bmatrix}$
		\State $t^{(k+1)}\leftarrow-R^{(k+1)} \widetilde{T}^\transpose \mathrm{\Omega} A^\transpose (A\mathrm{\Omega}A^\transpose)^\dagger $
		\vspace{0.2em}
		\State $X^{(k+1)}\leftarrow\begin{bmatrix}
		& t^{(k+1)} & R^{(k+1)}
		\end{bmatrix}$
		\EndFor
		\vspace{0.2em}
		\State \textbf{return} $\{X^{(k)}\}$
		\EndFunction
	\end{algorithmic}
\end{algorithm}

It is important to establish whether $\gpm$ and $\gpmo$ solve PGO. Empirically, we observe in the experiments that our proposed methods always converge to the global optima if the noise magnitudes are below a certain threshold. Theoretically, we can provide guarantees that $\gpm$ and $\gpmo$ converge to first-order critical points. Note that existing first- and second-order PGO methods in general need to choose stepsize carefully to guarantee the convergence to first-order critical points, whereas there is no stepsize tuning involved in either $\gpm$ or $\gpmo$.
\begin{theorem}\label{theorem::opt}
	\sloppy	Let $\{X^{(k)}\}$ be a sequence of iterates that is generated by either $\gpm$ or $\gpmo$. Then, 
	\begin{enumerate}[label={(\alph*)}]
		\item $F(X^{(k)})$ is non-increasing;\label{item::opt_1}
		\vspace{0.1em}
		\item $F(X^{(k)})\rightarrow F^\infty$ as $k\rightarrow \infty$;\label{item::opt_2}
		\vspace{0.1em}
		\item $\|X^{(k+1)}-X^{(k)}\|\rightarrow 0$ as $k\rightarrow \infty$ if $\nGamma\succ \nM$;\label{item::opt_3}
		\vspace{0.1em}
		\item $\mathrm{grad}\; F(X^{(k)})\rightarrow \0$ as $k\rightarrow\infty$ if $\nGamma\succ \nM$;\label{item::opt_4}
		\vspace{0.1em}
		\item $\|X^{(k+1)}-X^{(k)}\|\rightarrow 0$ as $k\rightarrow \infty$ if $\alpha > 0$;\label{item::opt_5}
		\vspace{0.1em}
		\item $\mathrm{grad}\; F(X^{(k)})\rightarrow \0$ as $k\rightarrow\infty$ if $\alpha > 0$. \label{item::opt_6}
	\end{enumerate}
\end{theorem}
\begin{proof}
	See \cite[Appendix B.3]{fan2019gpm}.
\end{proof}
\vspace{-0.75em}

\section{Accelerated Generalized Proximal Methods for PGO }\label{section::acc}
\vspace{-0.25em}

\begin{algorithm}[tb]
	\caption{The $\nagmo$ Method}
	\label{algorithm::nag*}
	\begin{algorithmic}[1]
		\State\textbf{Input}: An initial iterate $X^{(0)}=\begin{bmatrix}
		t^{(0)} &  R^{(0)}
		\end{bmatrix}\in\R^{d\times n}\times SO(d)^n$ and $X^{(-1)}=\begin{bmatrix}
		t^{(-1)} &  R^{(-1)}
		\end{bmatrix}\in\R^{d\times n}\times SO(d)^n$ in which $t^{(0)}=-R^{(0)} \widetilde{T}^\transpose \mathrm{\Omega} A^\transpose (A\mathrm{\Omega}A^\transpose)^\dagger $ and $t^{(-1)}=-R^{(-1)} \widetilde{T}^\transpose \mathrm{\Omega} A^\transpose (A\mathrm{\Omega}A^\transpose)^\dagger $, $s^{(0)}\in[1,\,+\infty)$, and the maximum number of iterations $N$.
		\vspace{0.25em}
		\State\textbf{Output}: A sequence of iterates $\{X^{(k)}, \, s^{(k)}\}$.\vspace{0.2em} 
		\Function { $\nagmo$}{$X^{(0)},\,X^{(-1)},\,s^{(0)},\, N$}
		\vspace{0.2em}
		\For{$k=0\rightarrow N-1$}
		\vspace{0.3em}
		\State $s^{(k+1)}\leftarrow\dfrac{\sqrt{4s^{(k)}{}^2+1}+1}{2}$,\quad $Y^{(k)}\leftarrow X^{(k)}+\dfrac{s^{(k)}-1}{s^{(k+1)}}\left(X^{(k)}-X^{(k-1)}\right)$
		\vspace{0.3em}
		\State $\begin{bmatrix}
		\theta_1^{(k)} & \cdots &\theta_n^{(k)} 
		\end{bmatrix}\leftarrow Y^{(k)}\mathrm{\Phi}$
		\vspace{0.25em}
		\For{$i=1\rightarrow n$}
		\vspace{0.2em}
		\State $R_i^{(k+1)}\leftarrow\arg\max\limits_{R_i\in SO(d) }\trace(R_i^\transpose\theta_i^{(k)})$
		\EndFor
		\vspace{0.2em}
		\State $R^{(k+1)}\leftarrow\begin{bmatrix}
		R_1^{(k+1)} & \cdots & R_n^{(k+1)}
		\end{bmatrix}$
		\State $t^{(k+1)}\leftarrow-R^{(k+1)} \widetilde{T}^\transpose \mathrm{\Omega} A^\transpose (A\mathrm{\Omega}A^\transpose)^\dagger $
		\vspace{0.25em}
		\State $X^{(k+1)}\leftarrow\begin{bmatrix}
		& t^{(k+1)} & R^{(k+1)}
		\end{bmatrix}$
		\EndFor
		\State \textbf{return} $\{X^{(k)}, \, s^{(k)}\}$
		\EndFunction
	\end{algorithmic}
\end{algorithm}

$\gpm$ and $\gpmo$ generalize proximal methods that use Euclidean gradients to update pose estimates, which is different from existing first-order PGO methods \cite{olson2006fast,grisetti2009nonlinear,tron2009distributed} using Riemannian gradients, and thus, it is possible to accelerate $\gpm$ and $\gpmo$ using \cite{nesterov1983method,nesterov2013introductory,ghadimi2016accelerated,jin2018accelerated,li2015accelerated}. 

Following Nesterov's accelerated proximal method \cite{nesterov1983method,nesterov2013introductory}, we might  extend $\gpmo$  to $\nagmo$ (\cref{algorithm::nag*}). $\nagmo$ is almost the same as $\gpmo$ at the beginning when $s^{(k)}$ is small but then more governed by the momentum term $X^{(k)}-X^{(k-1)}$ as $k$ increases. If we relax the constraints of $R_i\in SO(d)$ to any closed convex sets and choose initial iterate $X^{(0)}=X^{(-1)}$ and $s^{(0)}=1$, $\nagmo$ would converge to the global optima within  $O(1/N^2)$ time, whereas theoretically $\gpmo$ can not have a rate of convergence better than $O(1/N)$  \cite{nesterov1983method,nesterov2013introductory}. Even though PGO is a non-convex optimization problem, $\nagmo$ is expected to inherit the characteristics of Nesterov's accelerated proximal method and outperform $\gpmo$, and empirically, $\nagmo$ is indeed much faster than $\gpmo$. 

\begin{algorithm}[t]
	\caption{The $\agpmo$ Method}
	\label{algorithm::apm*}
	\begin{algorithmic}[1]
		\State\textbf{Input}: An initial iterate $X^{(0)}=\begin{bmatrix}
		t^{(0)} &  R^{(0)}
		\end{bmatrix}\in\R^{d\times n}\times SO(d)^n$ in which $t^{(0)}=-R^{(0)} \widetilde{T}^\transpose \mathrm{\Omega} A^\transpose (A\mathrm{\Omega}A^\transpose)^\dagger $, the maximum number of outer iterations $N$, the maximum number of inner iterations $N_0$, $\eta\in(0,\,1]\,$, and $\delta \in [0,\,\infty)$.
		\State\textbf{Output}: A sequence of iterates $\{X^{(k)}\}$.\vspace{0.2em} 
		\Function { $\agpmo$}{$X^{(0)},\, N,\, N_0,\, \delta$}
		\vspace{0.2em}
		\State $a^{(0)}\leftarrow1$,\quad$T^{(0)}\leftarrow X^{(0)}$,\quad $f^{(0)}\leftarrow F(X^{(0)})$\label{line::agpm_ini}
		\vspace{0.25em}
		\For{$k=0\rightarrow N-1$}
		\vspace{0.25em}
		\State $\{V^{(i)},\, s^{(i)}\}\leftarrow \nagmo(X^{(k)},\, T^{(k)},\, a^{(k)},\, N_0)$
		\vspace{0.35em}
		\If{$F(V^{(N_0)})\leq f^{(k)}- \delta\cdot \|V^{(N_0)}-X^{(k)}\|^2$}\label{line::agpm_delta}
		\vspace{0.35em}
		\State $a^{(k+1)}\leftarrow s^{(N_0)}$,\quad $X^{(k+1)}\leftarrow V^{(N_0)}$,\quad $T^{(k+1)}\leftarrow V^{(N_0-1)}$\label{line::algorithm4_nesterov}
		\vspace{0.1em}
		\Else
		\State $\{Z^{(i)}\}\leftarrow \gpmo(X^{(k)},\, N_0)$
		\vspace{0.25em}
		\State $a^{(k+1)}\leftarrow 1$,\quad $X^{(k+1)}\leftarrow Z^{(N_0)}$,\quad $T^{(k+1)}\leftarrow Z^{(N_0)}$\label{line::algorithm4_grad}
		\EndIf
		\vspace{0.25em}
		\State $f^{(k+1)}\leftarrow (1-\eta)\cdot f^{(k)} + \eta\cdot F(X^{(k+1)}) $
		\EndFor
		\vspace{0.2em}
		\State \textbf{return} $\{X^{(k)}\}$
		\EndFunction
	\end{algorithmic}
\end{algorithm}

Different from $\gpmo$, $\nagmo$ is not a descent algorithm, and might have ``Nesterov ripples'' due to high momentum term as $k$ increases \cite{o2015adaptive}. Moreover, even though $\nagmo$ is empirically much faster than $\gpmo$ to converge to first-order critical points, it seems difficult to have any theoretical guarantees of convergence for $\nagmo$. In order to address these theoretical and practical drawbacks of $\nagmo$, we propose $\agpmo$ (\cref{algorithm::apm*}) that is an extension of $\nagmo$ with adaptive restart -- a restart scheme is commonly used to improve the convergence of accelerated proximal methods  in convex optimization \cite{o2015adaptive}. In $\agpmo$, we implement $\gpmo$ for several iterations and then restart $\nagmo$ whenever the momentum term seems to take us in a bad direction, and as is shown later, though not a descent algorithm, $\agpmo$ is guaranteed to converge to first-order critical points under mild conditions. Since $\nagmo$ is usually preferred than $\gpmo$, it is recommended to choose a small $\delta$ in line~\ref{line::agpm_delta} of \cref{algorithm::apm*}. Besides acceleration, $\nagmo$  and $\agpmo$ are expected to escape saddle points faster than $\gpmo$ with some additional simple strategies adopted, for which interested readers can refer to \cite{jin2018accelerated} for more details. Similar to $\nagmo$ and $\agpmo$, we might also extend $\gpm$ to obtain $\mathsf{NAG-PGO}$ and $\agpm$, which are shown in \cite[Appendix A]{fan2019gpm}. 

In the experiments, we observe that $\agpm$ and $\agpmo$ converge to the global optima as long as $\gpm$ and $\gpmo$ converge to the global optima, however, $\agpm$ and $\agpmo$ are a lot faster than  $\gpm$ and $\gpmo$. Even though $\agpm$ and $\agpmo$ are not descent algorithms if $\eta < 1$, we still prove that $\agpm$ and $\agpmo$ converge to first-order critical points under mild conditions.
\begin{theorem}\label{theorem::nopt}
	\sloppy	Let $\{X^{(k)}\}$ be a sequence of iterates that is generated by either $\agpm$ or $\agpmo$. Then, 
	\begin{enumerate}[label={(\alph*)}]
		\item\label{item::agpm_1} $F(X^{(k)})\rightarrow F^\infty$ as $k\rightarrow \infty$;
		\vspace{0.1em}
		\item $\|X^{(k+1)}-X^{(k)}\|\rightarrow 0$ as $k\rightarrow \infty$ if $\nGamma\succ \nM$ and $\delta > 0$;\label{item::agpm_2}
		\vspace{0.1em}
		\item\label{item::agpm_3}  $\mathrm{grad}\; F(X^{(k)})\rightarrow \0$ as $k\rightarrow\infty$ if $\nGamma\succ \nM$, $\delta > 0$ and $N_0=1$;
		\vspace{0.2em}
		\item\label{item::agpm_4} $\|X^{(k+1)}-X^{(k)}\|\rightarrow 0$ as $k\rightarrow \infty$ if $\alpha > 0$ and $\delta > 0$;
		\vspace{0.1em}
		\item\label{item::agpm_5} $\mathrm{grad}\; F(X^{(k)})\rightarrow \0$ as $k\rightarrow\infty$ if $\alpha > 0$, $\delta > 0$ and $N_0=1$.
	\end{enumerate}
\end{theorem}
\begin{proof}
	See \cite[Appendix B.4]{fan2019gpm}.
\end{proof}

Note that even though Theorems \ref{theorem::nopt}\ref{item::agpm_3} and \ref{theorem::nopt}\ref{item::agpm_5} of  require the maximum number of inner iterations $N_0=1$ to guarantee $\grad F(\Xk{k})\rightarrow 0$, we still observe in the experiments that $\agpm$ and $\agpmo$ always converge to first-order critical points for any $N_0 >1$. As a result, it should be empirically all right to specify $N_0 >1$ so that the number of objective function evaluation is reduced and the overall efficiency is improved.

\vspace{-1em}

\section{Experiments}\label{section::experiments}
In this section, we evaluate the performance of generalized proximal methods for PGO that are proposed in \cref{section::fast,section::acc} on SLAM and distributed 3D sensor network localization, and make comparisons with existing techniques. All the tests have been performed on a Thinkpad P51 laptop with a 3.1GHz Intel Core Xeon that runs Ubuntu 18.04 and uses g++ 7.4 as C++ compiler.

\subsection{SLAM Benchmark Datasets}

In the first set of experiments, we implement $\agpmo$ (\cref{algorithm::apm*}) on a variety of popular 2D and 3D SLAM datasets and compare the results with $\sesync$ \cite{rosen2016se}, which is one of the fastest PGO methods.

For each of the dataset, we choose $\alpha = 0$, $N_0=10$, $\eta =1$ and $\delta = 1\times 10^{-5}$ for $\agpmo$, and $\agpmo$ terminates once the relative improvement of the objective function is less than $\epsilon=0.002$, i.e., $F(X^{(k)})\leq (1+\epsilon)F(X^{(k+1)})$. For $\sesync$, we use the default settings except the stopping criteria. In default, $\sesync$ does not stop until attaining a local optimum, whereas in our experiments, for a fair comparison, we terminate $\sesync$ once it achieves an equivalent accuracy as $\agpmo$, which takes less time than the default settings. For all the datasets, we use the chordal initialization \cite{carlone2015initialization} for both $\agpmo$ and $\sesync$. 

The results for these experiments are shown in \cref{table::2Dcomparison,table::3Dcomparison} and \cref{fig::slam2d,fig::slam3d,fig::slam_results}. In \cref{table::2Dcomparison,table::3Dcomparison}, $n$ is the number of unknown poses, $m$ is the number of edges, $f^*$ is the globally optimal objective value that can be obtained using $\sesync$, and $f$ is the objective value attained by $\agpmo$ and $\sesync$. In \cref{fig::slam2d,fig::slam3d}, we present the speed-up v.s. $\sesync$ and the relative objective error $(f-f^*)/f^*$ of $\agpmo$. In all the experiments, $\agpmo$ is several times faster than $\sesync$ to achieve modest accurate solutions with an average speed-up of  $5.14\mathrm{x}$ for 2D SLAM datasets and $9.14\mathrm{x}$ for 3D SLAM datasets. In addition, the average relative objective errors of $\agpmo$ for 2D and 3D SLAM datasets are $0.25\%$ and $0.075\%$, respectively, and such an accuracy is generally sufficient for practical use in SLAM. Furthermore, though not presented in this paper, $\agpmo$ converge to the global optima in all the experiments if enough computational time is given. 

We also compare the convergence of $\gpm$, $\gpmo$, $\agpm$ and $\agpmo$ with $\sesync$ on 2D and 3D SLAM datasets, whose results are shown in \cite[Appendix C]{fan2019gpm}.

\begin{table}[t]
	\renewcommand{\arraystretch}{1.3}
	\begin{tabular}{|c||c|c|c||c|c|c|c|c|}
		\hline			
		\multirow{2}{*}{Dataset}&\multirow{2}{*}{$n$}&\multirow{2}{*}{$m$}&\multirow{2}{*}{$f^*$} & \multicolumn{2}{c|}{$\sesync$\cite{rosen2016se}} & \multicolumn{2}{c|}{$\agpmo$ [ours]}\\
		\cline{5-8}
		&& && $f$  & Time (s) & $f$ & Time (s) \\
		\hline\hline
		\tt ais2klin&$15115$&$16727$&$1.885\times 10^2$&$1.885\times 10^2$&$4.83\times 10^{-1}$&$1.901\times 10^2$&$6.19\times10^{-2}$\\
		\hline
		\tt city &$10000$&$20687$&$6.386\times10^2$&$6.387\times 10^2$&$3.19\times 10^{-1}$&$6.388\times 10^2$&$4.02\times 10^{-2}$\\
		\hline
		\tt CSAIL     &$1045$&$1172$&$3.170\times 10^1$&$\ten{3.170}{1}$&$4.02\times10^{-3}$&$3.171\times 10^1$&$7.81\times10^{-4}$\\
		\hline
		\tt manhatta   &$3500$&$5453$&$6.432\times 10^3$&$\ten{6.434}{3}$&$\ten{9.02}{-3}$&$\ten{6.435}{3}$&$\ten{3.26}{-3}$\\
		\hline
		\tt intel     &$1728$&$2512$&$5.235\times 10^1$&$\ten{5.235}{1}$&$\ten{1.05}{-2}$&$\ten{5.248}{1}$&$\ten{5.10}{-3}$\\
		\hline
	\end{tabular}
	\vspace{0.75em}
	\caption{Results of the 2D SLAM datasets}	\label{table::2Dcomparison}
	
	\vspace{.5em}	
	\begin{tabular}{|c||c|c|c||c|c|c|c|c|}
		\hline			
		\multirow{2}{*}{Dataset}&\multirow{2}{*}{$n$}&\multirow{2}{*}{$m$}&\multirow{2}{*}{$f^*$} & \multicolumn{2}{c|}{$\sesync$ \cite{rosen2016se}} & \multicolumn{2}{c|}{$\agpmo$ [ours]}\\
		\cline{5-8}
		&& && $f$  & Time (s) & $f$ & Time (s) \\
		\hline\hline
		\tt cubicle  &$5750$&$16869$&$\ten{7.171}{2}$&$\ten{7.174}{2}$&$\ten{1.99}{-1}$&$\ten{7.178}{2}$&$\ten{6.64}{-2}$\\
		\hline
		\tt garage    &$1661$&$6275$&$\ten{1.263}{0}$&$\ten{1.263}{0}$&$\ten{4.38}{-1}$&$\ten{1.264}{0}$&$\ten{1.21}{-2}$\\
		\hline
		\tt grid     &$8000$&$22236$&$\ten{8.432}{4}$&$\ten{8.433}{4}$&$\ten{1.10}{0\hphantom{-}}$&$\ten{8.433}{4}$&$\ten{1.54}{-1}$\\
		\hline
		\tt rim     &$10195$&$29743$&$\ten{5.461}{3}$&$\ten{5.463}{3}$&$\ten{1.15}{0\hphantom{-}}$&$\ten{5.465}{3}$&$\ten{3.03}{-1}$\\
		\hline
		\tt sphere    &$2500$&$4949$&$\ten{1.687}{3}$&$\ten{1.687}{3}$&$\ten{1.74}{-1}$&$\ten{1.687}{3}$&$\ten{2.66}{-2}$\\
		\hline
		\tt torus     &$5000$&$9048$&$\ten{2.423}{4}$&$\ten{2.423}{4}$&$\ten{1.89}{-1}$&$\ten{2.425}{4}$&$\ten{5.50}{-2}$\\
		\hline
		\tt sphere-a    &$2200$&$8827$&$\ten{2.962}{6}$&$\ten{2.963}{6}$&$\ten{1.56}{-1}$&$\ten{2.963}{6}$&$\ten{4.06}{-2}$\\
		\hline
	\end{tabular}
	\vspace{0.75em}
	\caption{Results of the 3D SLAM datasets}	\label{table::3Dcomparison}
	\vspace{-1.5em}
\end{table}

\subsection{Distributed 3D Sensor Network Localization}
\vspace{-0.25em}
In the experiments of distributed 3D sensor network localization, it is assumed that the 3D sensor network is static and connected, and each node in the network can only communicate with its neighbours and can only measure the relative pose w.r.t. its neighbours, and as a result, we need to solve it distributedly to estimate poses of each node. As mentioned before, $\gpmo$ and $\agpmo$ always optimally recover the translation $t$, and thus, expect a faster convergence than $\gpm$ and $\agpm$. However, in distributed PGO, similar to second-order PGO methods \cite{kaess2012isam2,rosen2014rise,kuemmerle11icra,rosen2016se} solving linear systems to evaluate the descent direction, $\gpmo$ and $\agpmo$ have to use iterative solvers to solve \cref{eq::tsolf} to recover the translation $t$, which usually reduces the efficiency of optimization. In contrast, $\gpm$ and $\agpm$ only need the estimated poses of itself and its neighbours in optimization and do not have to solve linear systems as $\gpmo$ and $\agpmo$ and second-order PGO methods \cite{kaess2012isam2,rosen2014rise,kuemmerle11icra,rosen2016se}. Therefore, $\gpm$ and $\agpm$  are well suited for the distributed PGO without inducing extra efforts. Furthermore, there is no loss of theoretical guarantees for $\gpm$ and $\agpm$ in distributed PGO. 
\begin{figure}[t]
	\vspace{-2em}
	\centering
	\begin{tabular}{cc}
		\hspace{4.5mm}\subfloat[][]{\includegraphics[trim =20mm 6mm 1mm 0mm,width=0.45\textwidth]{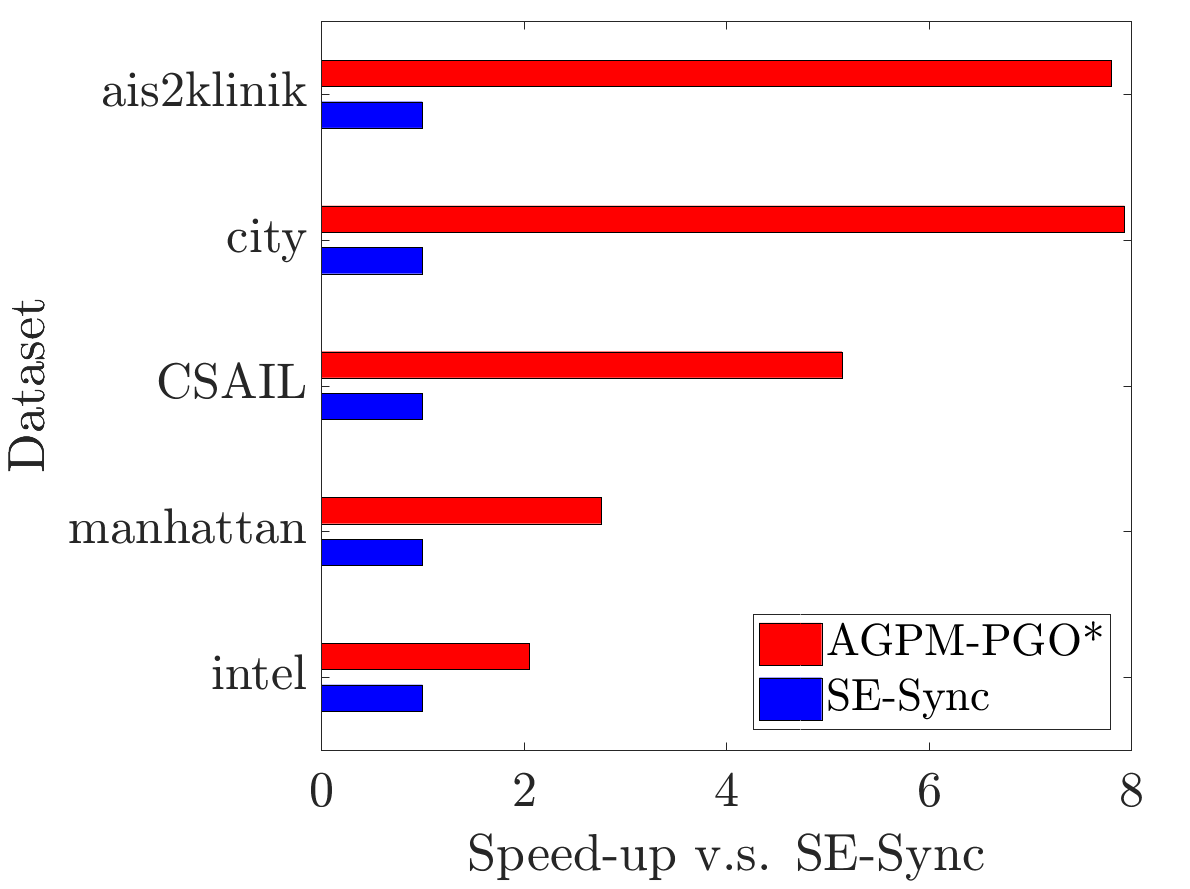}}\hspace{2em} &
		\subfloat[][]{\includegraphics[trim =20mm 6mm 1mm 0mm,width=0.45\textwidth]{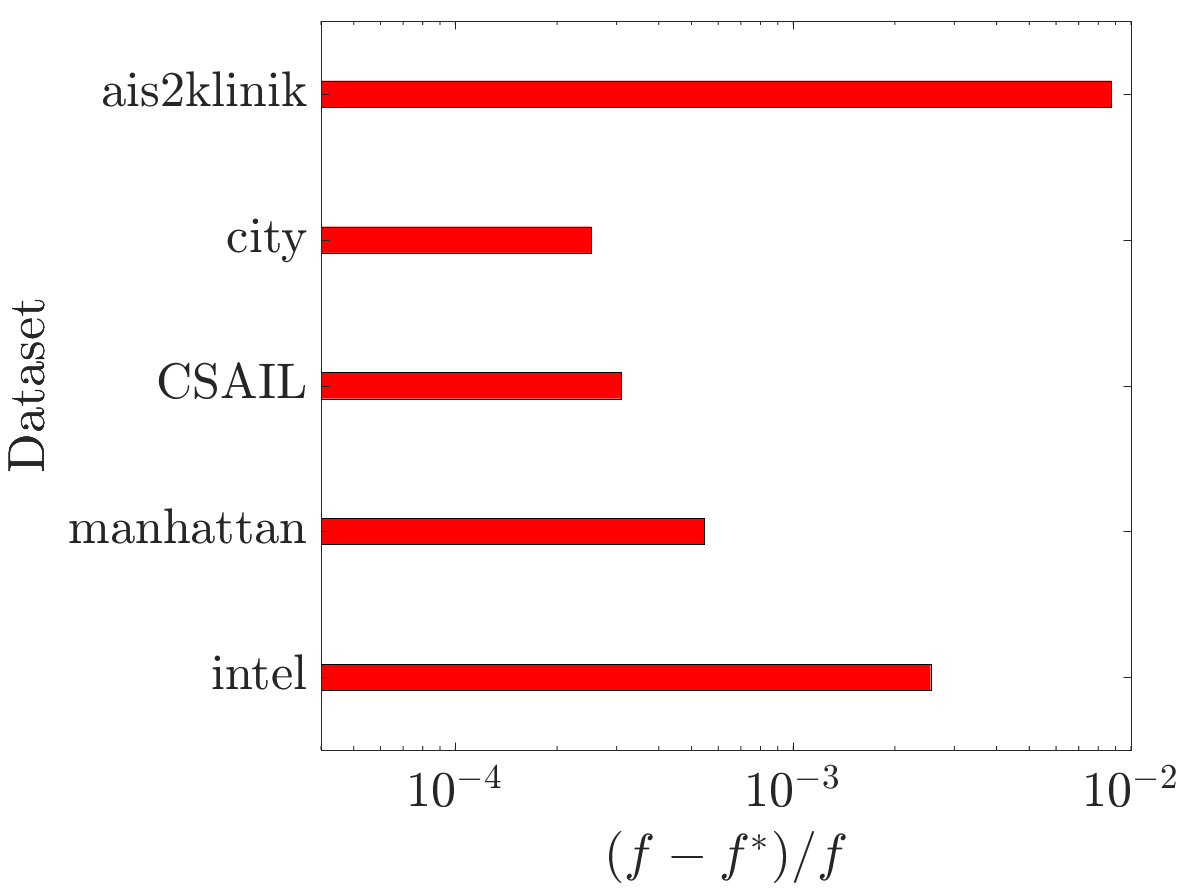}}
	\end{tabular}
	\vspace{-0.6em}
	\caption{The results of $\agpmo$ for 2D SLAM datasets. The results are (a) speed-up v.s. $\sesync$ \cite{rosen2016se} and (b) the relative objective error  $(f-f^*)/f^*$ of $\agpmo$. For 2D SLAM datasets, the average speed up is $5.14\mathrm{x}$  and the average relative objective error  is $0.25\%$.}
	\label{fig::slam2d} 
	\vspace{-0.35em}
	\begin{tabular}{cc}
		\hspace{4.5mm}\subfloat[][]{\includegraphics[trim =20mm 6mm 1mm 0mm,width=0.45\textwidth]{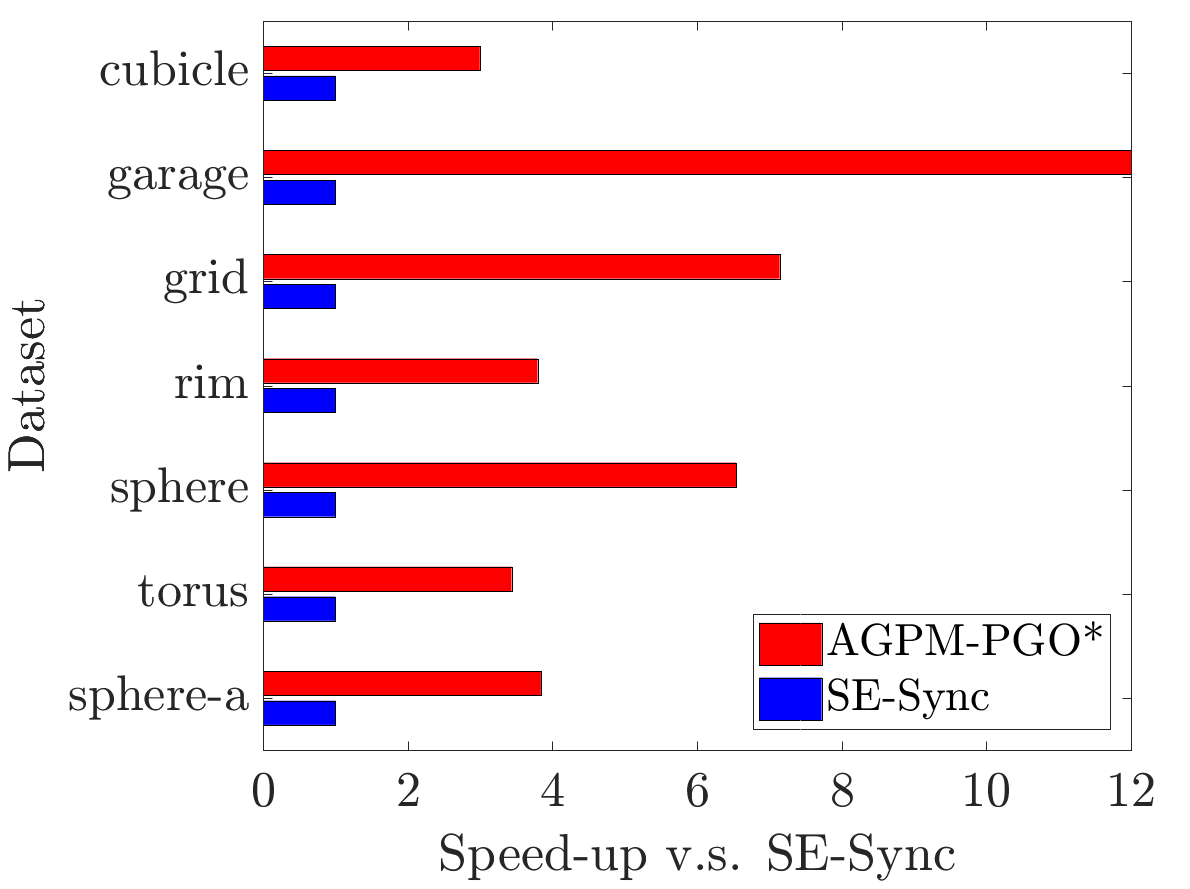}}\hspace{2em} &
		\subfloat[][]{\includegraphics[trim =20mm 6mm 1mm 0mm,width=0.45\textwidth]{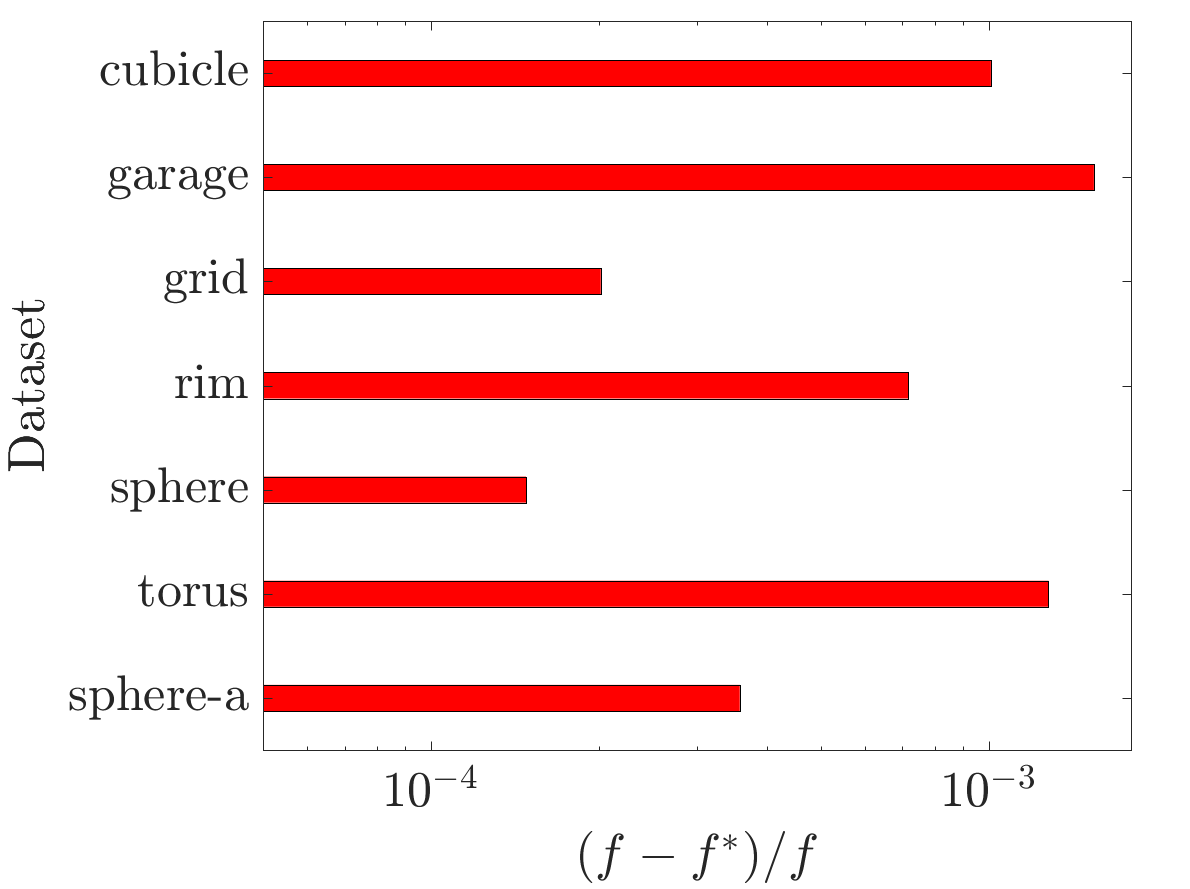}}
	\end{tabular}
	\vspace{-0.6em}
	\caption{The results of $\agpmo$ for 3D SLAM datasets. The results are (a) speed-up v.s. $\sesync$ \cite{rosen2016se} and (b) the relative objective error  $(f-f^*)/f^*$ of $\agpmo$. For 3D SLAM datasets, the average speed up is $9.14\mathrm{x}$  and the average relative objective error  is $0.075\%$.}
	\label{fig::slam3d}
	\vspace{-1em} 
\end{figure}

In the experiments, we simulate a distributed 3D sensor network on an ellipsoid with $n=200$ vertices (nodes) and $m=600$ edges. The noisy measurements $\tilde{g}_{ij}=(\nR_{ij},\,\nt_{ij})\in SE(3)$ are generated according to the model $\nR_{ij}=\underline{R}{}_{ij}\exp(\xi_{ij}^R)$ and $\nt_{ij}=\underline{t}{}_{ij}+\xi_{ij}^t$, in which $\xi_{ij}^R\sim \mathcal{N}(0,\, \sigma_R\cdot \I)$ with $\delta_R = 0.05\,\mathrm{rad}$ with $\xi_{ij}^t\sim \mathcal{N}(0,\, \sigma_t\cdot \I)$ and $\delta_t = 0.05\,\mathrm{m}$. We compute the statistics over $30$ runs and use the chordal initialization\footnote{The chordal initialization can be distributedly solved by relaxing PGO as convex quadratic programming.} for all the runs. The results are shown in \cref{fig::sensor_networks}. In all the 30 runs, $\agpm$ converge to the global optima with an average rotation error of $0.0253$ rad and an average relative translation error of $1.60\%$.

\section{Conclusions}\label{section::conclusion}
\vspace{-0.5em}
In this paper, we have proposed generalized proximal methods for PGO and proved that our proposed methods converge to first-order critical points. In addition, we have accelerated the rates of convergence without loss of any theoretical guarantees. Our proposed methods can be distributed and parallelized with minimal efforts and with no compromise of efficiency. In the experiments, our proposed methods are much faster than existing techniques to converge to modest accuracy that is sufficient for practical use. Though not presented in this paper, our proposed methods can also be extended for incremental smoothing \cite{fan2019incremental}.

\begin{figure}[!htpb]
	\vspace{-2em}
	\centering
	\begin{tabular}{ccc}
		\subfloat[][\tt city10000]{\includegraphics[trim =0mm 0mm 0mm 0mm,width=0.23\textwidth]{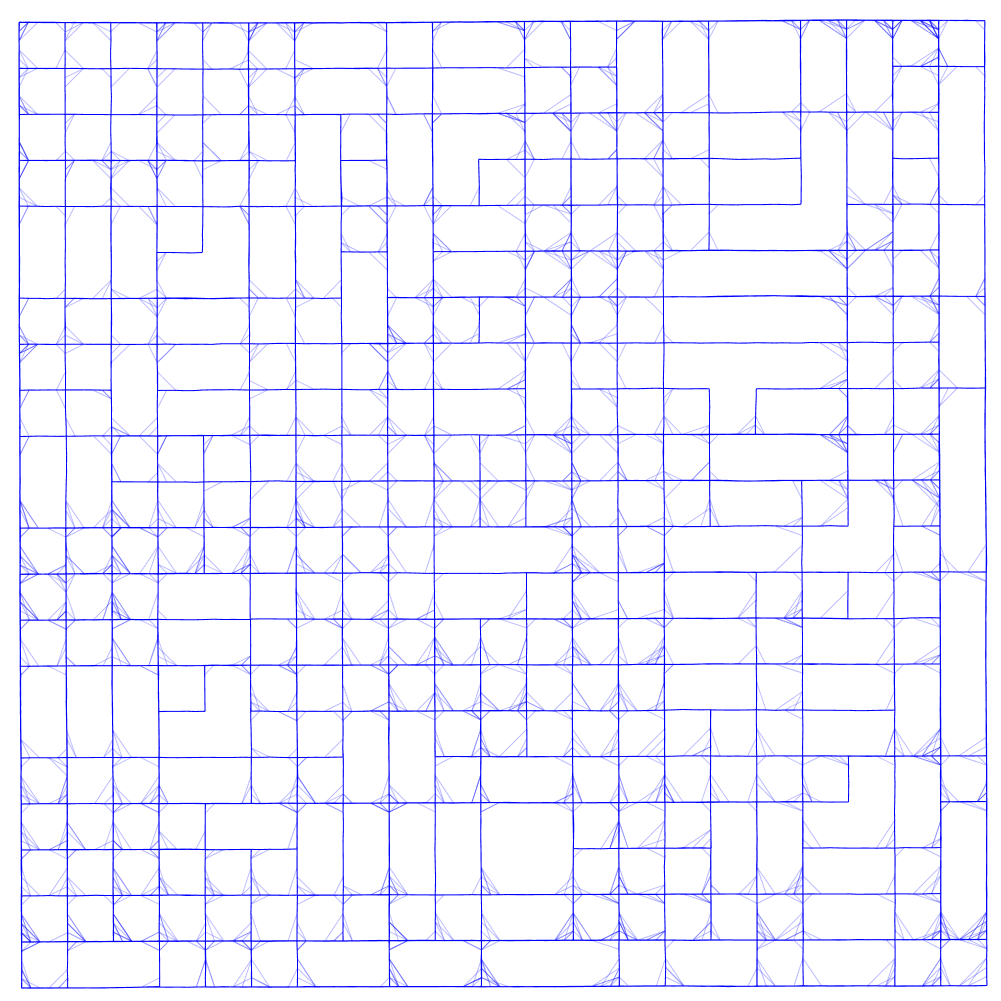}}\hspace{1em} &
		\subfloat[][\tt intel]{\includegraphics[trim =0mm 0mm 0mm 0mm,width=0.25\textwidth]{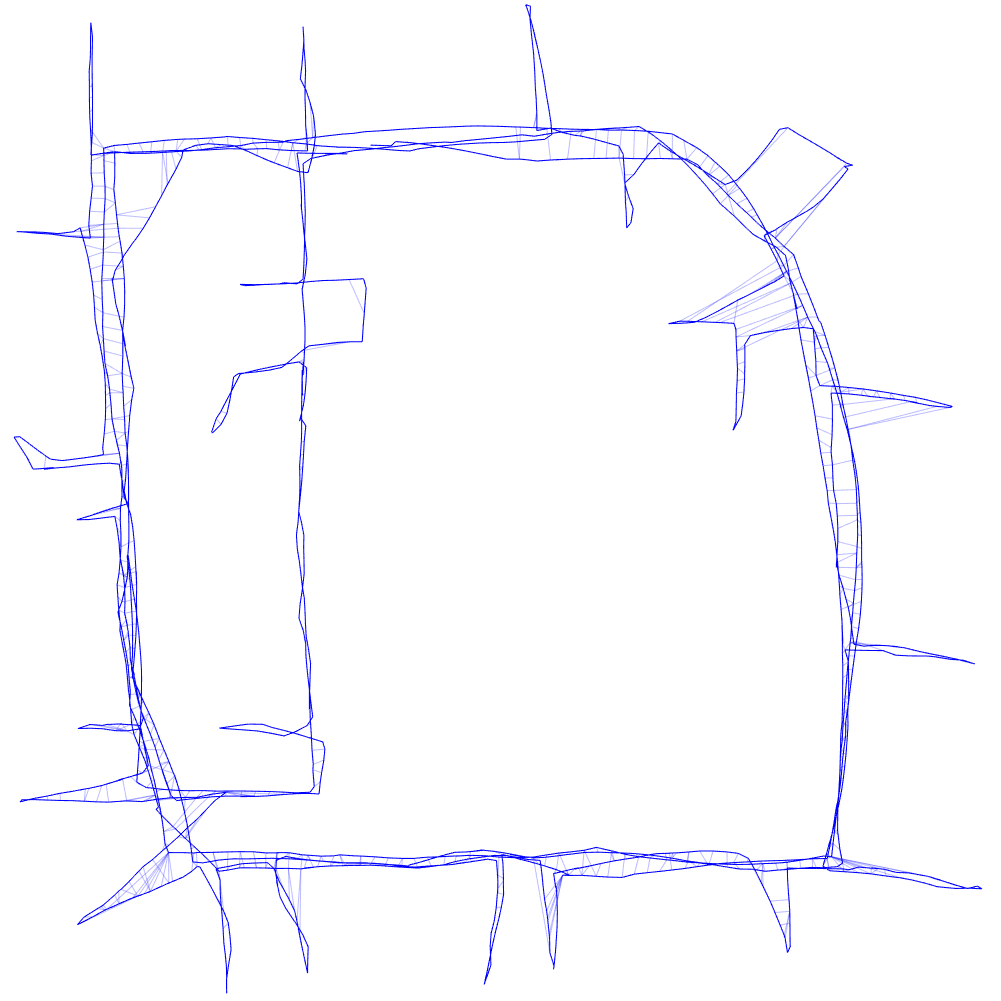}} &
		\subfloat[][\tt ais2klinik]{\includegraphics[trim =0mm 0mm 0mm 0mm,width=0.25\textwidth]{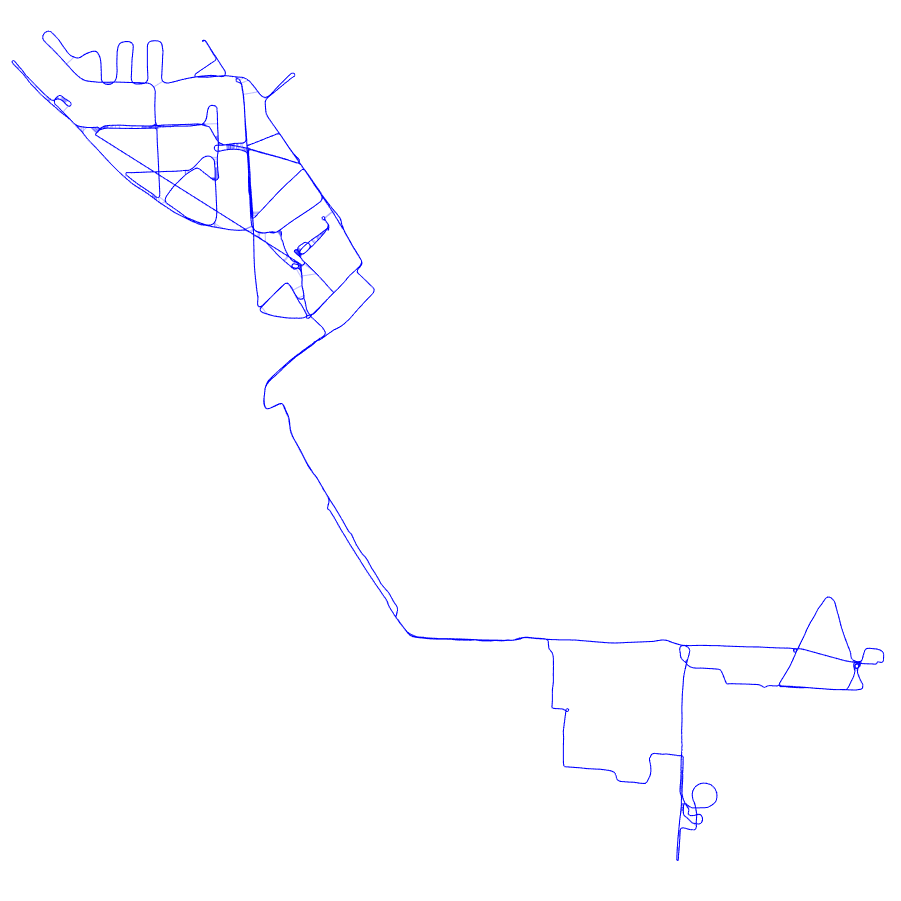}} \vspace{-1.em}\\
		\subfloat[][\tt garage]{\includegraphics[trim =0mm 0mm 0mm 0mm,width=0.28\textwidth]{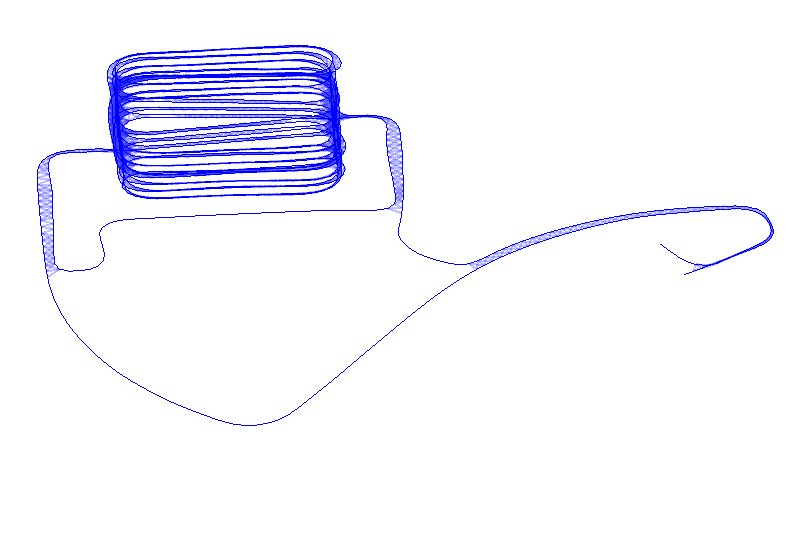}} &
		\subfloat[][\tt rim]{\includegraphics[trim =0mm 0mm 0mm 0mm,width=0.28\textwidth]{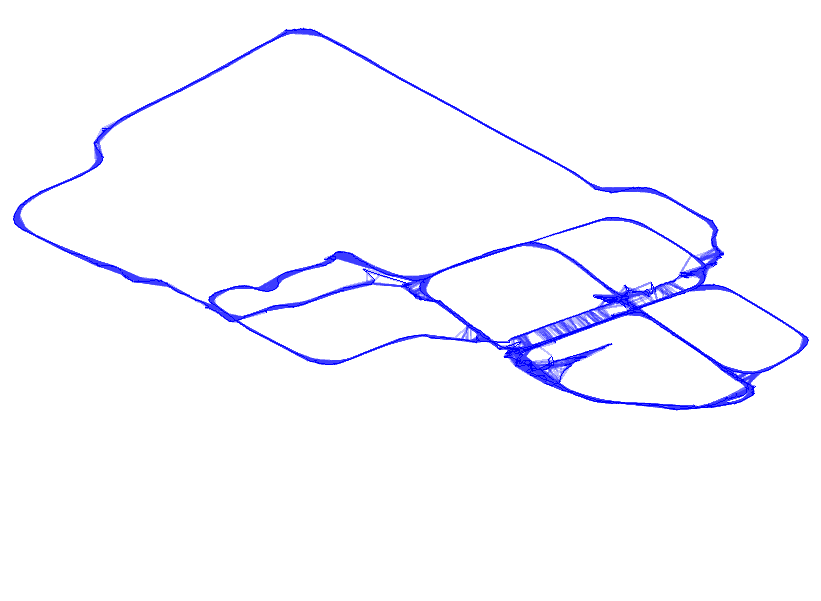}} &
		\subfloat[][\tt sphere]{\includegraphics[trim =0mm 0mm 0mm 0mm,width=0.23\textwidth]{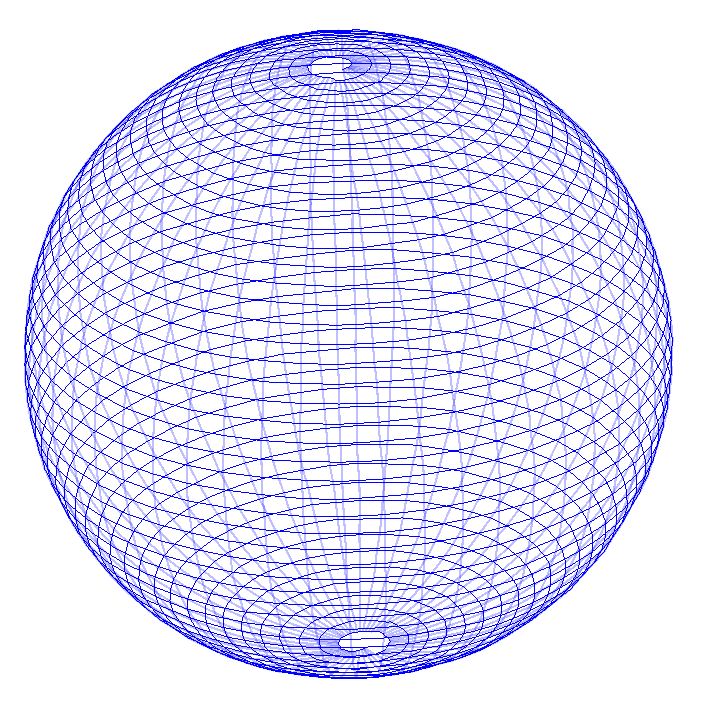}}
	\end{tabular}
	\caption{The results of $\agpmo$ on some 2D and 3D SLAM benchmark datasets. }\label{fig::slam_results} 
	\vspace{-0.5em}
	\begin{tabular}{cc}
		\hspace{5mm}\subfloat[][]{\includegraphics[trim =20mm 2mm 1mm 0mm,width=0.45\textwidth]{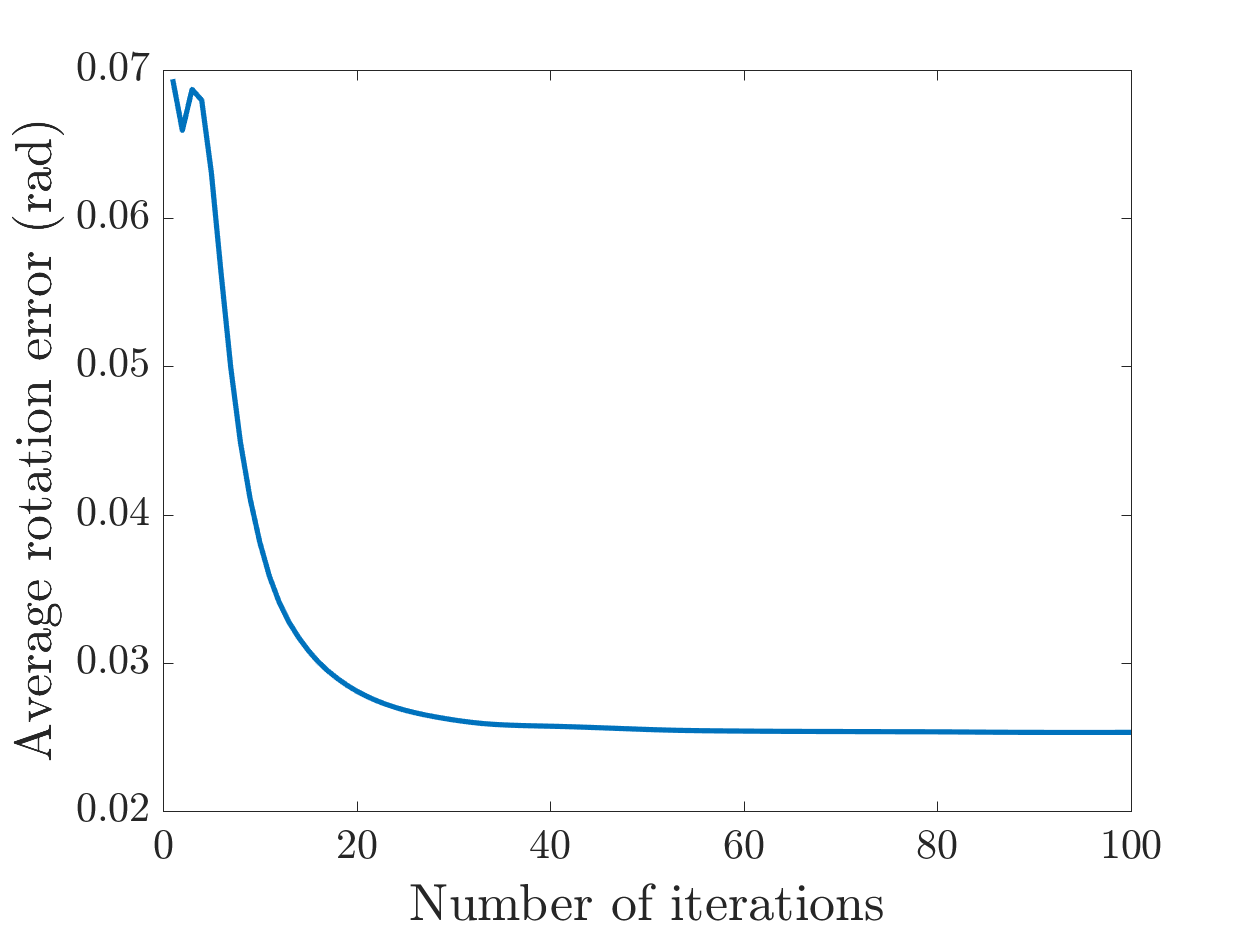}}\hspace{2em} &
		\subfloat[][]{\includegraphics[trim =20mm 2mm 1mm 0mm,width=0.45\textwidth]{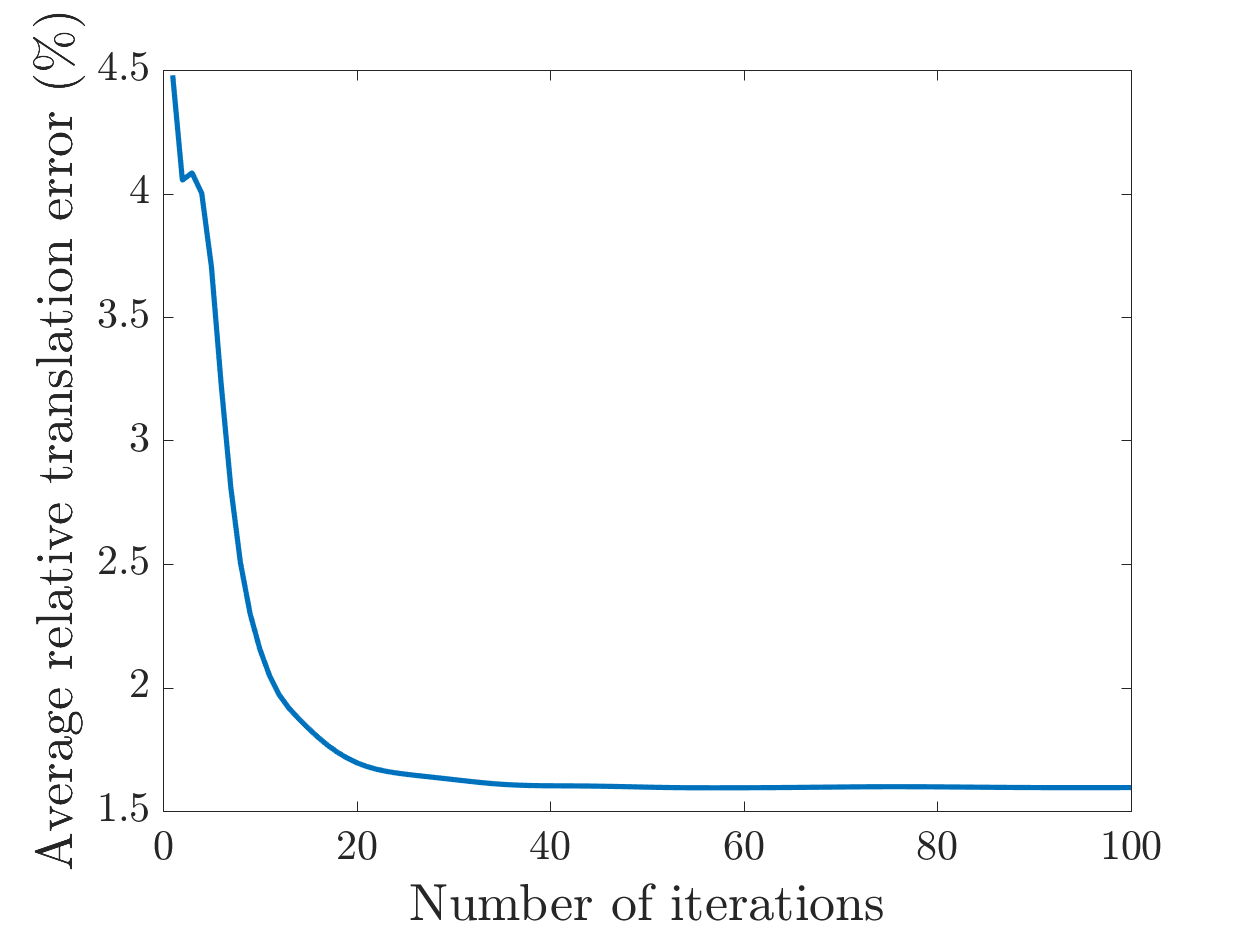}}
	\end{tabular}
	\vspace{-0.5em}
	\caption{The results of distributed sensor network localization using $\agpm$ over 30 runs. The results are (a) average rotation error and (b) average relative translation error. In all the 30 runs, $\agpm$ converge to global optima.}
	\label{fig::sensor_networks} 
	\vspace{-1.5em}
\end{figure}

\vspace{-1em}
\bibliographystyle{unsrt}
\bibliography{mybib}

\begin{thebibliography}{10}

\bibitem{cadena2016past}
Cesar Cadena, Luca Carlone, Henry Carrillo, Yasir Latif, Davide Scaramuzza,
  Jos{\'e} Neira, Ian Reid, and John~J Leonard.
\newblock Past, present, and future of simultaneous localization and mapping:
  Toward the robust-perception age.
\newblock {\em IEEE Transactions on robotics}, 2016.

\bibitem{singer2011three}
Amit Singer and Yoel Shkolnisky.
\newblock Three-dimensional structure determination from common lines in
  cryo-em by eigenvectors and semidefinite programming.
\newblock {\em SIAM journal on imaging sciences}, 4(2):543--572, 2011.

\bibitem{tron2009distributed}
Roberto Tron and Ren{\'e} Vidal.
\newblock Distributed image-based {3-D} localization of camera sensor networks.
\newblock In {\em IEEE Conference on Decision and Control (CDC)}, 2009.

\bibitem{olson2006fast}
Edwin Olson, John Leonard, and Seth Teller.
\newblock Fast iterative alignment of pose graphs with poor initial estimates.
\newblock In {\em IEEE International Conference on Robotics and Automation
  (ICRA)}, 2006.

\bibitem{grisetti2009nonlinear}
Giorgio Grisetti, Cyrill Stachniss, and Wolfram Burgard.
\newblock Nonlinear constraint network optimization for efficient map learning.
\newblock {\em IEEE Transactions on Intelligent Transportation Systems},
  10(3):428--439, 2009.

\bibitem{kaess2012isam2}
Michael Kaess, Hordur Johannsson, Richard Roberts, Viorela Ila, John~J Leonard,
  and Frank Dellaert.
\newblock {iSAM2}: {Incremental} smoothing and mapping using the bayes tree.
\newblock {\em The International Journal of Robotics Research}, 2012.

\bibitem{rosen2014rise}
David Rosen, Michael Kaess, and John Leonard.
\newblock {RISE}: An incremental trust-region method for robust online sparse
  least-squares estimation.
\newblock {\em IEEE Transactions on Robotics}, 2014.

\bibitem{kuemmerle11icra}
R.~Kuemmerle, G.~Grisetti, H.~Strasdat, K.~Konolige, and W.~Burgard.
\newblock g2o: A general framework for graph optimization.
\newblock In {\em Proceedings of the IEEE International Conference on Robotics
  and Automation (ICRA)}, pages 3607--3613, Shanghai, China, May 2011.

\bibitem{rosen2016se}
David Rosen, Luca Carlone, Afonso Bandeira, and John Leonard.
\newblock {SE-Sync}: {A certifiably correct algorithm for synchronization over
  the special Euclidean group}.
\newblock {\em arXiv preprint arXiv:1612.07386}, 2016.

\bibitem{fan2019iros}
Taosha Fan, Hanlin Wang, Michael Rubenstein, and Todd Murphey.
\newblock Efficient and guaranteed planar pose graph optimization using the
  complex number representation.
\newblock In {\em IEEE/RSJ International Conference on Intelligent Robots and
  Systems (IROS)}, 2019.

\bibitem{nesterov1983method}
Yurii Nesterov.
\newblock A method for unconstrained convex minimization problem with the rate
  of convergence {O} (1/k\^{} 2).
\newblock In {\em Doklady AN USSR}, volume 269, pages 543--547, 1983.

\bibitem{nesterov2013introductory}
Yurii Nesterov.
\newblock {\em Introductory lectures on convex optimization: {A} basic course},
  volume~87.
\newblock Springer Science \& Business Media, 2013.

\bibitem{ghadimi2016accelerated}
Saeed Ghadimi and Guanghui Lan.
\newblock Accelerated gradient methods for nonconvex nonlinear and stochastic
  programming.
\newblock {\em Mathematical Programming}, 156(1-2):59--99, 2016.

\bibitem{jin2018accelerated}
Chi Jin, Praneeth Netrapalli, and Michael~I Jordan.
\newblock Accelerated gradient descent escapes saddle points faster than
  gradient descent.
\newblock In {\em Conference On Learning Theory}, 2018.

\bibitem{li2015accelerated}
Huan Li and Zhouchen Lin.
\newblock Accelerated proximal gradient methods for nonconvex programming.
\newblock {\em Advances in neural information processing systems}, 28:379--387,
  2015.

\bibitem{parikh2014proximal}
Neal Parikh, Stephen Boyd, et~al.
\newblock Proximal algorithms.
\newblock {\em Foundations and Trends{\textregistered} in Optimization},
  1(3):127--239, 2014.

\bibitem{fan2019gpm}
Taosha Fan and Todd Murphey.
\newblock Generalized proximal methods for pose graph optimization.
\newblock [Online]. Available: \url{https://arxiv.org/abs/2012.02709}.

\bibitem{umeyama1991least}
Shinji Umeyama.
\newblock Least-squares estimation of transformation parameters between two
  point patterns.
\newblock {\em IEEE Transactions on Pattern Analysis \& Machine Intelligence},
  (4):376--380, 1991.

\bibitem{mcadams2011computing}
Aleka McAdams, Andrew Selle, Rasmus Tamstorf, Joseph Teran, and Eftychios
  Sifakis.
\newblock Computing the singular value decomposition of 3x3 matrices with
  minimal branching and elementary floating point operations.
\newblock Technical report, University of Wisconsin-Madison Department of
  Computer Sciences, 2011.

\bibitem{o2015adaptive}
Brendan O’donoghue and Emmanuel Candes.
\newblock Adaptive restart for accelerated gradient schemes.
\newblock {\em Foundations of computational mathematics}, 15(3):715--732, 2015.

\bibitem{carlone2015initialization}
Luca Carlone, Roberto Tron, Kostas Daniilidis, and Frank Dellaert.
\newblock Initialization techniques for {3D SLAM}: a survey on rotation
  estimation and its use in pose graph optimization.
\newblock In {\em IEEE International Conference on Robotics and Automation
  (ICRA)}, 2015.

\bibitem{fan2019incremental}
Taosha Fan and Todd~D Murphey.
\newblock Fast incremental smoothing using generalized proximal methods.
\newblock In preparation.

\bibitem{horn2012matrix}
Roger~A Horn and Charles~R Johnson.
\newblock {\em Matrix analysis}.
\newblock Cambridge university press, 2012.

\bibitem{absil2009optimization}
P-A Absil, Robert Mahony, and Rodolphe Sepulchre.
\newblock {\em Optimization algorithms on matrix manifolds}.
\newblock Princeton University Press, 2009.

\end{thebibliography}

\newpage
\setcounter{section}{0}
\renewcommand\thesection{\Alph{section}}
\numberwithin{equation}{section}

\section{The $\nagm$ and $\agpm$ Methods}\label{appendix::alg}
\vspace{-0.5em}
\subsection{The $\nagm$ Method}
\vspace{-1.75em}
\begin{algorithm}[!htpb]
	\caption{The $\nagm$ Method}
	\label{algorithm::nag}
	\begin{algorithmic}[1]
		\State\textbf{Input}: An initial iterate $X^{(0)}=\begin{bmatrix}
		t^{(0)} &  R^{(0)}
		\end{bmatrix}\in\R^{d\times n}\times SO(d)^n$ and $X^{(-1)}=\begin{bmatrix}
		t^{(-1)} &  R^{(-1)}
		\end{bmatrix}\in\R^{d\times n}\times SO(d)^n$, $s^{(0)}\in[1,\,+\infty)$, and the maximum number of iterations $N$.
		\State\textbf{Output}: A sequence of iterates $\{X^{(k)}, \, s^{(k)}\}$.\vspace{0.2em} 
		\Function { $\nagm$\,}{$X^{(0)},\,X^{(-1)},\,s^{(0)},\, N$}
		\For{$k=0\rightarrow N-1$}
		\State $s^{(k+1)}\leftarrow\dfrac{\sqrt{4s^{(k)}{}^2+1}+1}{2}$,\quad $Y^{(k)}\leftarrow X^{(k)}+\dfrac{s^{(k)}-1}{s^{(k+1)}}\left(X^{(k)}-X^{(k-1)}\right)$
		\State $\begin{bmatrix}
		\theta_1^{(k)} & \cdots &\theta_n^{(k)} 
		\end{bmatrix}\leftarrow Y^{(k)}\mathrm{\Phi}$
		\For{$i=1\rightarrow n$}
		\vspace{0.15em}
		\State $R_i^{(k+1)}\leftarrow\arg\max\limits_{R_i\in SO(d) }\trace(R_i^\transpose\theta_i^{(k)})$
		\vspace{-0.15em}
		\EndFor
		\State $R^{(k+1)}\leftarrow\begin{bmatrix}
		R_1^{(k+1)} & \cdots & R_n^{(k+1)}
		\end{bmatrix}$
		\State $t^{(k+1)}\leftarrow  R^{(k+1)}\mathrm{\Xi}+X^{(k)}\mathrm{\Psi}$
		\State $X^{(k+1)}\leftarrow\begin{bmatrix}
		& t^{(k+1)} & R^{(k+1)}
		\end{bmatrix}$
		\EndFor
		\State \textbf{return} $\{X^{(k)}, \, s^{(k)}\}$
		\EndFunction
	\end{algorithmic}
\end{algorithm}
\vspace{-2em}
\subsection{The $\agpm$ Method}
\vspace{-1.5em}
\begin{algorithm}[!htpb]
	\caption{The $\agpm$ Method}
	\label{algorithm::apm}
	\begin{algorithmic}[1]
		\State\textbf{Input}: An initial iterate $X^{(0)}=\begin{bmatrix}
		t^{(0)} &  R^{(0)}
		\end{bmatrix}\in\R^{d\times n}\times SO(d)^n$, the maximum number of outer iterations $N$, the maximum number of inner iterations $N_0$, $\eta\in(0,\,1]\,$, and $\delta \in [0,\,\infty)$.
		\State\textbf{Output}: A sequence of iterates $\{X^{(k)}\}$.\vspace{0.2em} 
		\Function { $\agpm$}{$X^{(0)},\, N,\, N_0,\, \delta$}
		\vspace{0.2em}
		\State $a^{(0)}\leftarrow1$,\quad$T^{(0)}\leftarrow X^{(0)}$,\quad $f^{(0)}\leftarrow F(X^{(0)})$
		\For{$k=0\rightarrow N-1$}
		\vspace{0.25em}
		\State $\{V^{(i)},\, s^{(i)}\}\leftarrow \nagmo(X^{(k)},\, T^{(k)},\, a^{(k)},\, N_0)$
		\vspace{0.25em}
		\If{$F(V^{(N_0)})\leq f^{(k)}- \delta\cdot \|V^{(N_0)}-X^{(k)}\|^2$}
		\vspace{0.25em}
		\State $a^{(k+1)}\leftarrow s^{(N_0)}$,\quad $X^{(k+1)}\leftarrow V^{(N_0)}$,\quad $T^{(k+1)}\leftarrow V^{(N_0-1)}$
		\vspace{-0.25em}
		\Else
		\State $\{Z^{(i)}\}\leftarrow \gpmo(X^{(k)},\, N_0)$
		\vspace{0.25em}
		\State $a^{(k+1)}\leftarrow 1$,\quad $X^{(k+1)}\leftarrow Z^{(N_0)}$,\quad $T^{(k+1)}\leftarrow Z^{(N_0)}$
		\EndIf
				\vspace{0.15em}
		\State $f^{(k+1)}\leftarrow (1-\eta)\cdot f^{(k)} + \eta\cdot F(X^{(k+1)}) $
		\EndFor
		\State \textbf{return} $\{X^{(k)}\}$
		\EndFunction
	\end{algorithmic}
\end{algorithm}
\newpage
\section{Proofs}
\subsection{Proof of \cref{theorem::first-order}}\label{appendix::proof1}
If we let 
\begin{equation}\label{eq::FR}
F_{ij}^R(X)=\frac{1}{2}\|R_i\nR_{ij}-R_j\|^2
\end{equation}
and
\begin{equation}\label{eq::Ft}
F_{ij}^t(X)=\frac{1}{2}\|R_i\nt_{ij}+t_i-t_j\|^2,
\end{equation}
then we obtain
\begin{subequations}\label{eq::dFR}
\begin{equation}
\nabla_{R_i} F_{ij}^R(X)=R_i-R_j\nR_{ij}^\transpose,
\end{equation}
\begin{equation}
\nabla_{R_j} F_{ij}^R(X)=R_j-R_i\nR_{ij},
\end{equation}
\end{subequations} 
and
\begin{subequations}\label{eq::dFt}
\begin{equation}
\nabla_{t_i} F_{ij}^t(X)=R_i\nt_{ij}+t_i-t_j,
\end{equation}
\begin{equation}
\nabla_{R_i} F_{ij}^t(X)=\big(R_i\nt_{ij}+t_i-t_j\big)\nt_{ij}^\transpose,
\end{equation}
\begin{equation}
\nabla_{t_j} F_{ij}^t(X)=t_j-R_i\nt_{ij}-t_i.
\end{equation}
\end{subequations} 
Note that $F_{ij}^R(X)$ and  $F_{ij}^t(X)$ only depend on $t_i$, $R_i$, $t_j$ and $R_j$, and as a result, $\nabla F_{ij}^R(X)$ and  $\nabla F_{ij}^t(X)$ are well defined by \cref{eq::dFR,eq::dFt}, respectively. Then, from \cref{eq::dFR,eq::dFt,eq::FR,eq::Ft}, it is straightforward to show that
\begin{equation}\label{eq::proxFR}
\begin{aligned}
&\|R_i\nR_{ij}-\frac{1}{2}\Rk{k}_i\nR_{ij}-\frac{1}{2}\Rk{k}_j\|^2+\|R_j-\frac{1}{2}\Rk{k}_i\nR_{ij}-\frac{1}{2}\Rk{k}_j\|^2\\
=&\|R_i-\Rk{k}_i\|^2+\|R_j-\Rk{k}_j\|^2+\trace\Big(\big(R_i-\Rk{k}_i\big)^\transpose\nabla_{R_i}F_{ij}^R(\Xk{k})\Big)+\\
&\trace\Big(\big(R_j-\Rk{k}_j\big)^\transpose\nabla_{R_j}F_{ij}^R(\Xk{k})\Big)+F_{ij}^R(\Xk{k})\\
=&\|R_i-\Rk{k}_i\|^2+\|R_j-\Rk{k}_j\|^2+\\
&\trace\Big(\big(X-\Xk{k}_i\big)^\transpose\nabla F_{ij}^R(\Xk{k})\Big)+F_{ij}^R(\Xk{k}),
\end{aligned}
\end{equation}
and
\begin{equation}\label{eq::proxFt}
\begin{aligned}
&\|R_i\nt_{ij}+t_i-\frac{1}{2}\Rk{k}_i\nt_{ij}-\frac{1}{2}\tk{k}_i-\frac{1}{2}\tk{k}_j\|^2+\|t_j-\frac{1}{2}\Rk{k}_i\nt_{ij}-\frac{1}{2}\tk{k}_i-\frac{1}{2}\tk{k}_j\|^2\\
=&\|(R_i-\Rk{k}_i)\nt_{ij}+t_i-\tk{k}_i\|^2+\|t_j-\tk{k}_j\|^2+\big(t_i-\tk{k}_i\big)^\transpose\nabla_{t_i}F_{ij}^t(\Xk{k})+\\
&\trace\Big(\big(R_i-\Rk{k}_i\big)^\transpose\nabla_{R_i}F_{ij}^t(\Xk{k})\Big)+\big(t_j-\tk{k}_j\big)^\transpose\nabla_{t_j}F_{ij}^t(\Xk{k})+F_{ij}^t(\Xk{k})\\
=&\|(R_i-\Rk{k}_i)\nt_{ij}+t_i-\tk{k}_i\|^2+\|t_j-\tk{k}_j\|^2+\\
&\trace\Big(\big(X-\Xk{k}_i\big)^\transpose\nabla F_{ij}^t(\Xk{k})\Big)+F_{ij}^t(\Xk{k}),
\end{aligned}
\end{equation}
It is by definition that
\begin{equation}\label{eq::F}
F(X)=\sum_{(i,\,j)\in\aEE}\left(\kappa_{ij}\cdot F_{ij}^R(X)+\tau_{ij}\cdot F_{ij}^t(X)\right),
\end{equation}
and 
\begin{equation}\label{eq::dF}
\nabla F(X)=\sum_{(i,\,j)\in\aEE}\left(\kappa_{ij}\cdot \nabla F_{ij}^R(X)+\tau_{ij}\cdot\nabla F_{ij}^t(X)\right).
\end{equation}
 Substitute \cref{eq::proxFR,eq::proxFt} into \cref{eq::objn} and simplify the resulting equation with \cref{eq::F,eq::dF}, the result is
\begin{multline}\label{eq::prox_f}
F(\Xk{k})+\trace\left(({X-\Xk{k}})^\transpose{\nabla F(\Xk{k})}\right)+\\\frac{1}{2}\trace\left(\big(X-\Xk{k}\big) \nH\big(X-\Xk{k}\big)^\transpose\right),
\end{multline}
in which $\nOmega\triangleq\begin{bmatrix}
\mathrm{\Omega}^\tau & \nOmega^\nu{}^\transpose\\
\nOmega^\nu & \nOmega^\rho\hphantom{{}^\transpose}
\end{bmatrix}$ with $\nOmega^{\tau}=\diag\{\nOmega_1^\tau,\,\cdots,\,\nOmega_n^\tau\}\in \R^{n\times n}$, $\nOmega^\rho=\diag\{\nOmega_1^\rho,\,\cdots,\,\nOmega_n^\rho\}\in \R^{dn\times dn}$ and $\nOmega^\nu=\diag\{
\nOmega_1^\nu,\,  \cdots ,\, \nOmega_n^\nu \}\in \R^{dn\times n}$, in which 
\begin{subequations}
	\begin{equation}\label{eq::Otau}
	\mathrm{\Omega}_i^\tau = \sum_{(i,\,j)\in\EE}2\cdot\tau_{ij} \in \R,
	\end{equation}
	\begin{equation}\label{eq::Orho}
	\nOmega_i^\rho = \sum\limits_{(i,\,j)\in \EE}2\cdot \kappa_{ij}\cdot\I+\sum\limits_{(i,\,j)\in \aEE}2\cdot\tau_{ij}\cdot\nt_{ij}\nt_{ij}^\transpose\in \R^{d\times d},
	\end{equation}
	\begin{equation}\label{eq::Onu}
	\nOmega_i^\nu = \sum\limits_{(i,\,j)\in \aEE}2\cdot\tau_{ij}\cdot \nt_{ij}\in \R^{d}.
	\end{equation}
\end{subequations}
The proof is completed.

\subsection{Proof of \cref{theorem::bound}}\label{appendix::proof2}
\subsubsection{Proof of \ref{item::bound_1}}
It is without loss of any generality to reformulate \cref{eq::obj_quad} as  
\begin{multline}\label{eq::obj_expand}
F(X)=F(\Xk{k})+\trace\left(({X-\Xk{k}})^\transpose{\nabla F(\Xk{k})}\right)+\\
\frac{1}{2}\trace\left(\big(X-\Xk{k}\big) \nM\big(X-\Xk{k}\big)^\transpose\right).
\end{multline}
Note that \eqref{eq::prox_f} is an upper bound of $F(X)$ and \cref{eq::obj_expand}, and as a result, we obtain $\nOmega\succeq \nM$, which completes the proof of \ref{item::bound_1}. 

\subsubsection{Proof of \ref{item::bound_2}}
From \cref{eq::Onu,eq::Orho,eq::Otau}, if we reorder $X=\begin{bmatrix}
t_1 & \cdots & t_n & R_1 & \cdots & R_n
\end{bmatrix}$ to $X'=\begin{bmatrix}
t_1 & R_1 & t_2 & R_2 \cdots & t_n & R_n
\end{bmatrix}$, then $\nOmega$ is accordingly reordered to a  block diagonal $\nOmega'\triangleq\diag\{\nOmega'_1,\,\cdots,\,\nOmega_n'\}\in \R^{(d+1)n\times (d+1)n}$, in which $\nOmega_i'\in \R^{(d+1)\times (d+1)}$ are the principal minors of $2\cdot\nM$. Let $\lambda_{\max}(\nOmega_i')$, $\lambda_{\max}(\nOmega)$, $\lambda_{\max}(\nM)$, etc., be the greatest eigenvalue of corresponding matrices. As a result of Courant-Fischer theorem \cite[Theorem 4.2.6]{horn2012matrix}, it is straightforward to show $\lambda_{\max}(\nOmega_i')\leq 2\cdot\lambda_{\max}(\nM)$, from which we further obtain $\lambda_{\max}(\nOmega)=\lambda_{\max}(\nOmega')\leq 2\cdot\lambda_{\max}(\nM)$. Then, for any $c\in \R$, if $\dfrac{c}{2}\cdot\I \succeq \nM$, we obtain $c\geq 2\cdot\lambda_{\max}(\nM)$, and thus, $c\geq \lambda_{\max}(\nOmega)$ and $c\cdot\I\succeq \nOmega$, which completes the proof of \ref{item::bound_2}. 

\subsection{Proof of \cref{theorem::opt}}\label{appendix::proof3}
\subsubsection{Proof of \ref{item::opt_1}}
For $\gpm$, it should be noted that we define $G(X|\Xk{k})$ as
\begin{multline}\label{eq::prox_F}
G(X|\Xk{k})=F(\Xk{k})+\trace\left(({X-\Xk{k}})^\transpose{\nabla F(\Xk{k})}\right)+\\\frac{1}{2}\trace\left(\big(X-\Xk{k}\big) \nGamma\big(X-\Xk{k}\big)^\transpose\right)
\end{multline}
in \cref{eq::obj_first2}, with which \cref{eq::obj_first2} is equivalent to
\begin{equation}\label{eq::minG}
\Xk{k+1}=\arg\min_{X\in \R^{d\times n}\times SO(d)^n} G(X|\Xk{k}).
\end{equation}
Since $G(X|\Xk{k})$ is an upper bound of $F(X)$ that attains the same value with $F(X)$ at $\Xk{k}$ and $\Xk{k+1}$ minimizes $G(X|\Xk{k})$, it can be concluded that
\begin{equation}\label{eq::Fd1}
F(\Xk{k+1})\leq G(\Xk{k+1}|\Xk{k})\leq G(\Xk{k}|\Xk{k})= F(\Xk{k}),
\end{equation} 
which suggests that $\gpm$ is non-increasing. 

For $\gpmo$, if we substitute 
\begin{equation}\label{eq::tsolB}
t=-R  \widetilde{V}^\transpose L(W^\tau)^\dagger
\end{equation}
into \cref{eq::obj_quad} and simplify the resulting equation, we obtain
\begin{equation}\label{eq::obj_quadR}
\min_{R\in  SO(d)^n} F'(R)\triangleq\dfrac{1}{2}\trace(R \nM_{\!\!R}  R^\transpose),
\end{equation}
in which 
\begin{equation}
\nM_{\!\!R}= L(\widetilde{G}^\rho) +\widetilde{\mathrm{\Sigma}}^\rho - \widetilde{V}^\transpose L(W^\tau)^\dagger\widetilde{V}.
\end{equation}
For each iterate in $\gpmo$, from \cref{eq::tsolf}, it is straight to show that $F'(R^{(k)})=F(X^{(k)})$, $\nabla F'(R^{(k)})=\nabla_R F(X^{(k)})$ and $\nabla_t F(X^{(k)})=\0$, and as a result, $G(X|X^{(k)})$ can be simplified to 
\begin{multline}\label{eq::GxR}
G(X|X^{(k)})=F'(R^{(k)}) + \trace\left(\nabla F'(R^{(k)})^\transpose(R-R^{(k)})\right)+\\
\frac{1}{2}\trace\left(\big(X-\Xk{k}\big) \nGamma\big(X-\Xk{k}\big)^\transpose\right).
\end{multline}
If we substitute \cref{eq::tsol} into \cref{eq::GxR} and marginalize out $t\in \R^{d\times n}$, we obtain 
\begin{multline}\label{eq::GR}
G'(R|R^{(k)})=F'(R^{(k)}) + \trace\left(\nabla F'(R^{(k)})^\transpose(R-R^{(k)})\right)+\\
\frac{1}{2}\trace\left(\big(R-\Rk{k}\big) \nGamma'\big(R-\Rk{k}\big)^\transpose\right),
\end{multline}
in which
$$\nGamma'=\nGamma^\rho - {\nGamma^\nu}{}^\transpose{\rGamma^\tau}{}^{-1} {\nGamma^\nu} $$
is positive semidefinte, and it should be noted that $$G'(R^{(k)}|R^{(k)})=F'(R^{(k)})=F(X^{(k)}).$$ 
Furthermore, \cref{eq::obj_R} is equivalent to
\begin{equation}\label{eq::GRsol}
R^{(k+1)} = \arg\min_{R\in SO(d)^n} G'(R|R^{(k)}), 
\end{equation}
and it is by definition that
\begin{equation}\label{eq::Fd2}
F(X^{(k+1)})= F'(R^{(k+1)})\leq G'(R^{(k+1)}|R^{(k)}) \leq F'(R^{(k)})=F(X^{(k)}),
\end{equation}
which suggests that $\gpmo$ is non-increasing.

From \cref{eq::Fd1,eq::Fd2}, it can be concluded that $F(X^{(k)})$ is non-increasing, which completes the proof of \ref{item::opt_1}.

\subsubsection{Proof of \ref{item::opt_2}}
From \cref{theorem::opt}\ref{item::opt_1}, it is known that $F(\Xk{k})$ is non-increasing. Furthermore, $F(X)$ is bounded below, and as a result, there exists $F^\infty\in \R$ such that  $F(\Xk{k})\rightarrow F^\infty$, which completes the proof of \ref{item::opt_2}. 
\subsubsection{Proof of \ref{item::opt_3}}
For $\gpm$, it is from \cref{eq::prox_F,eq::minG} that
\begin{multline}
\nonumber
F(\Xk{k})+\trace\Big(\big(\Xk{k+1}-\Xk{k}\big)^\transpose{\nabla F(\Xk{k})}\Big)+\\
\frac{1}{2}\trace\left(\big(\Xk{k+1}-\Xk{k}\big) \nGamma\big(\Xk{k+1}-\Xk{k}\big)^\transpose\right)\leq F(\Xk{k}),
\end{multline}
and from \cref{eq::obj_expand}, $F(\Xk{k+1})$ is equivalent to
\begin{multline}
\nonumber
F(\Xk{k+1})= F(\Xk{k})+\trace\Big(\big(\Xk{k+1}-\Xk{k}\big)^\transpose{\nabla F(\Xk{k})}\Big)+\\
\frac{1}{2}\trace\left(\big(\Xk{k+1}-\Xk{k}\big) \nM\big(\Xk{k+1}-\Xk{k}\big)^\transpose\right),
\end{multline}
which implies 
\begin{multline}
\label{eq::diff_F}
F(\Xk{k+1})-F(\Xk{k})\leq \\
\frac{1}{2}\trace\left(\big(\Xk{k+1}-\Xk{k}\big) (\nM-\nGamma)\big(\Xk{k+1}-\Xk{k}\big)^\transpose\right).
\end{multline}
If $\nGamma\succ\nM$, there exists $\epsilon >0$ such that $\nGamma \succeq \nM + \epsilon\cdot\I$. Then, from \cref{eq::diff_F}, it can be concluded that
\begin{equation}\label{eq::improve}
F(\Xk{k+1})-F(\Xk{k})\leq -\epsilon\cdot\|\Xk{k+1}-\Xk{k}\|^2.
\end{equation}

For $\gpmo$, it is from \cref{eq::GR,eq::GRsol} that 
\begin{multline}
\nonumber
F'(\Rk{k})+\trace\Big(\big(\Rk{k+1}-\Rk{k}\big)^\transpose{\nabla F'(\Xk{k})}\Big)+\\
\frac{1}{2}\trace\left(\big(\Rk{k+1}-\Rk{k}\big) \nGamma'\big(\Rk{k+1}-\Rk{k}\big)^\transpose\right)\leq F'(\Rk{k}),
\end{multline}
and from \cref{eq::obj_quadR}, $F(\Rk{k+1})$ is equivalent to
\begin{multline}
\nonumber
F'(\Rk{k+1})= F'(\Rk{k})+\trace\Big(\big(\Rk{k+1}-\Rk{k}\big)^\transpose{\nabla F'(\Rk{k})}\Big)+\\
\frac{1}{2}\trace\left(\big(\Rk{k+1}-\Rk{k}\big) \nM'\big(\Rk{k+1}-\Rk{k}\big)^\transpose\right),
\end{multline}
which implies
\begin{multline}
\label{eq::diff_FR}
F'(\Rk{k+1})-F'(\Rk{k})\leq \\
\frac{1}{2}\trace\left(\big(\Rk{k+1}-\Rk{k}\big) (\nM'-\nGamma')\big(\Rk{k+1}-\Rk{k}\big)^\transpose\right).
\end{multline}
From $\nGamma \succ \nM\succeq 0$, it is straightforward to show that $\nGamma'\succ\nM'$ and there exists $\epsilon_R > 0$ such that $\nGamma' \succeq \nM'+\epsilon_R\cdot\I$, and as a result, we obtain
\begin{equation}\label{eq::improve1}
F'(\Rk{k+1})-F'(\Rk{k})\leq -\epsilon_R\cdot\|\Rk{k+1}-\Rk{k}\|^2.
\end{equation}
Furthermore, from \cref{eq::tsolf,eq::improve1}, it can be shown that there exists $\epsilon >0$ such that
\begin{equation}\label{eq::dR}
\epsilon_R\cdot\|\Rk{k+1}-\Rk{k}\|^2\geq \epsilon\cdot\|\Xk{k+1}-\Xk{k}\|^2.
\end{equation}
As a result of $F(\Xk{k})=F'(\Rk{k})$ and $F(\Xk{k+1})=F'(\Rk{k+1})$ and \cref{eq::improve1,eq::dR}, we obtain
\begin{equation}\label{eq::improve2}
F(\Xk{k+1})-F(\Xk{k})\leq -\epsilon\cdot\|\Xk{k+1}-\Xk{k}\|^2.
\end{equation}

From \cref{theorem::opt}\ref{item::opt_2}, we obtain $F(\Xk{k})\rightarrow F^\infty$. Then, as a result of \cref{eq::improve,eq::improve2}, it can be shown that for $\gpm$ and $\gpmo$, $\|\Xk{k+1}-\Xk{k}\|\rightarrow 0$ as $k\rightarrow\infty$, which completes the proof of \ref{item::opt_3}.

\subsubsection{Proof of \ref{item::opt_4}}
For $\gpm$, from Riemannian optimization \cite{absil2009optimization,rosen2016se}, if we assume that the Euclidean gradient is $\nabla F(X)=\begin{bmatrix}
\nabla_t F(X) &\;&\nabla_R F(X)
\end{bmatrix}$,
in which $\nabla_t F(X)\in \R^{d\times n}$ and $\nabla_R F(X)\in \R^{d\times dn}$ correspond to the translation $t=\begin{bmatrix}
t_1 & \cdots & t_n
\end{bmatrix}\in \R^{d\times n}$ and the rotation 
$R=\begin{bmatrix}
R_1 & \cdots & R_n
\end{bmatrix}\in SO(d)^n$, respectively, then the Riemannian gradient $\grad F(X)$ can be computed as
\begin{equation}\label{eq::gradFX}
\grad F(X)=\begin{bmatrix}
\grad_t F(X) &\;& \grad_R F(X)
\end{bmatrix},
\end{equation}
in which
\begin{equation}\label{eq::gradFt}
\grad_t F(X)=\nabla_t F(X)
\end{equation} 
corresponds to the translation $t$, and
\begin{equation}\label{eq::gradFR}
\grad_R F(X)= \nabla_R F(X)-R\;\mathrm{SymBlockDiag}_d(R^\transpose\nabla_R F(X))
\end{equation}
corresponds to the rotation $R$. In \cref{eq::gradFR}, $\mathrm{SymBlockDiag}_d:\R^{dn\times dn}\rightarrow \R^{dn\times dn}$ is a linear operator
\begin{equation}
\nonumber
\mathrm{SymBlockDiag}_d(Z)\triangleq\frac{1}{2}\mathrm{BlockDiag}_d(Z+Z^\transpose),
\end{equation}
in which $\mathrm{BlockDiag}_d:\R^{dn\times dn}\rightarrow \R^{dn\times dn}$ is also a a linear operator that extracts the $(d\times d)$-block diagonals of a matrix, i.e.,
\begin{equation}
\nonumber
\mathrm{BlockDiag}_d(Z)\triangleq\begin{bmatrix}
Z_{11} & & \\
&\ddots &\\
& & Z_{nn}
\end{bmatrix}.
\end{equation} 
As a result, \cref{eq::gradFX,eq::gradFt,eq::gradFR} result in a linear operator $\QQ_X:\R^{d\times (d+1)n}\rightarrow\R^{d\times (d+1)n} $ that depends on $X$ such that the Riemannian gradient $\grad F(X)$ and the Euclidean gradient ${\nabla F(X)}$ are related as
\begin{equation}\label{eq::gradF}
\grad F(X) = \QQ_X\big(\nabla F(X)\big).
\end{equation}
Similarly, for $G(X|\Xk{k})$ in \cref{eq::prox_F}, the Riemannian gradient $\grad G(X|\Xk{k})$ is
\begin{equation}
\nonumber
\begin{aligned}
\grad G(X|\Xk{k})=&\QQ_X\big(\nabla F(\Xk{k})\big)+\QQ_X\big((X-\Xk{k})\nGamma\big)\\
=&\QQ_{X}\big(\nabla F(X)\big)+\QQ_{X}\big(\nabla F(\Xk{k})-\nabla F(X)\big)+\\
&\QQ_X\big((X-\Xk{k})\nGamma\big)\\
=&\grad F(X)+\QQ_{X}\big(\nabla F(\Xk{k})-\nabla F(X)\big)+\\
&\QQ_X\big((X-\Xk{k})\nGamma\big),
\end{aligned}
\end{equation}
and thus, we obtain
\begin{equation}\label{eq::gradG}
\begin{aligned}
&\grad G(\Xk{k+1}|\Xk{k})\\
=&\grad F(\Xk{k+1})+\QQ_{\Xk{k+1}}\big(\nabla F(\Xk{k})-\nabla F(\Xk{k+1})\big)+\\
&\QQ_{\Xk{k+1}}\big((\Xk{k+1}-\Xk{k})\nGamma\big)
\end{aligned}
\end{equation}
Note that $\Xk{k+1}$ minimizes $G(X|\Xk{k})$, and thus, $\grad G(\Xk{k+1}|\Xk{k})=\0$ always holds, with which \cref{eq::gradG} suggests
\begin{multline}\label{eq::gradFx}
	\grad F(\Xk{k+1})
	=\QQ_{\Xk{k+1}}\big(\nabla F(\Xk{k+1})-\nabla F(\Xk{k})\big)+\\
	\QQ_{\Xk{k+1}}\big((\Xk{k}-\Xk{k+1})\nGamma\big).
\end{multline}
Then, it can be shown that
\begin{multline}\label{eq::gradFnorm}
	\|\grad F(\Xk{k+1})\|\leq\|\QQ_{\Xk{k+1}}\big(\nabla F(\Xk{k})-\nabla F(\Xk{k+1})\big)\|+\\
	\quad\quad\hspace{7em}\|\QQ_{\Xk{k+1}}\big((\Xk{k+1}-\Xk{k})\nGamma\big)\| \\
	\leq \|\QQ_{\Xk{k+1}}\|\cdot\|\nabla F(\Xk{k})-\nabla F(\Xk{k+1}\|+\\
	\|\QQ_{\Xk{k+1}}\|_2\cdot\|\nGamma\|_2\cdot\|\Xk{k+1}-\Xk{k}\|,
\end{multline}
in which $\|\cdot\|_2$ denotes the induced 2-norm of linear operators. It is known that $\nabla F(X)$ is Lipschitz continuous, then there exists $L>0$ such that
\begin{equation}\label{eq::QFk}
	\|\nabla F(\Xk{k})-\nabla F(\Xk{k+1})\|\leq  L\cdot\|\Xk{k+1}-\Xk{k}\|.
\end{equation}
From \cref{eq::QFk,eq::gradFnorm}, we obtain
\begin{equation}\label{eq::gradFxnorm}
\|\grad F(\Xk{k+1})\|\leq \|\QQ_{\Xk{k+1}}\|_2\cdot(L+\|\nGamma\|_2)\cdot\|\Xk{k+1}-\Xk{k}\|.
\end{equation}
From \cref{eq::gradFX,eq::gradFt,eq::gradFR}, it can be seen that $\QQ_X$ only depends on $R\in SO(d)^n$ and $SO(d)^n$ is a compact manifold, and thus, $\|\QQ_{X}\|_2$ is bounded. Furthermore, \cref{theorem::opt}\ref{item::opt_3} indicates that $\|\Xk{k+1}-\Xk{k}\|\rightarrow 0$ if $\nGamma\succ \nM$, from which and the equation above, it can be concluded that
\begin{equation}\label{eq::gradFxnorm1}
\|\grad F(\Xk{k+1})\|\rightarrow 0
\end{equation}
as $k\rightarrow\infty$ if $\nGamma\succ \nM$.

For $\gpmo$,  the Riemannian gradient of $F'(R)$ in \cref{eq::obj_quadR} is 
\begin{equation}\label{eq::gradF_R}
\grad F'(R)= \nabla F'(R)-R\;\mathrm{SymBlockDiag}_d(R^\transpose\nabla F'(R)).
\end{equation}
From \cref{eq::GR}, we obtain
\begin{equation}
\nonumber
\begin{aligned}
\grad G'(R|\Rk{k})=&\QQ_R\big(\nabla F'(\Rk{k})\big)+\QQ_R\big((R-\Rk{k})\nGamma'\big)\\
=&\QQ_{R}\big(\nabla F'(R)\big)+\QQ_{R}\big(\nabla F'(\Rk{k})-\nabla F'(R)\big)+\\
&\QQ_R\big((R-\Rk{k})\nGamma'\big)\\
=&\grad F'(R)+\QQ_{R}\big(\nabla F'(\Rk{k})-\nabla F'(R)\big)+\\
&\QQ_R\big((R-\Rk{k})\nGamma'\big),
\end{aligned}
\end{equation}
in which $\QQ_{R}:\R^{d\times dn}\rightarrow\R^{d\times dn}$ is the linear operator defined by \cref{eq::gradF_R}, and we obtain
\begin{equation}
\begin{aligned}
\grad G'(\Rk{k+1}|\Rk{k})=&\grad F'(\Rk{k+1})+\QQ_{\Rk{k+1}}\big(\nabla F'(\Rk{k})-\nabla F'(\Rk{k+1})\big)+\\
&\QQ_{\Rk{k+1}}\big((\Rk{k+1}-\Rk{k})\nGamma'\big).
\end{aligned}
\end{equation}
Similar to $\gpm$, it can be shown that
\begin{equation}\label{eq::gradGk'}
\|\grad F'(\Rk{k+1})\|\rightarrow0.
\end{equation}
Moreover, as mentioned in the proof of \ref{item::opt_1}, we have $\nabla F'(R^{(k)})=\nabla_R F(X^{(k)})$ and $\nabla_t F(X^{(k)})=\0$, which further indicates that
$\|\grad F(\Xk{k})\|\rightarrow 0$.
\subsubsection{Proof of \ref{item::opt_5} and \ref{item::opt_6}}
Note that $\nGamma=\alpha\cdot\I +\nOmega \succ \nM$ if $\alpha >0$. Then, from \ref{item::opt_3} and \ref{item::opt_4} of \cref{theorem::opt}, it can be concluded that \ref{item::opt_5} and \ref{item::opt_6} hold as long as $\alpha > 0$, which completes the proof of \ref{item::opt_5} and \ref{item::opt_6}.

\subsection{Proof of \cref{theorem::nopt}}\label{appendix::proof4}
\subsubsection{Proof of \ref{item::agpm_1}}
Even though \cref{algorithm::apm*} is not necessarily a descent algorithm, we can still prove $f^{(k+1)}\leq f^{(k)}$ and $F(\Xk{k})\leq f^{(k)}$ by induction. Note that $f^{(0)}=F{(\Xk{0})} $, and thus, $F{(\Xk{0})}\leq f^{(0)} $. If $\Xk{k+1}$ is generated from $\nagm$ or $\nagmo$, we obtain $F(\Xk{k+1})\leq f^{(k)}$; otherwise, $\Xk{k+1}$ is generated from $\gpm$ or $\gpmo$, and according to \cref{theorem::opt}, we obtain $F{(\Xk{k+1})}\leq F{(\Xk{k})}$, from which it can be further shown that  $F{(\Xk{k+1})}\leq f^{(k)}$ as long as $F{(\Xk{k})}\leq f^{(k)} $. If $F(\Xk{k+1})\leq f^{(k)}$, we obtain $f^{(k+1)}=(1-\eta)\cdot f^{(k)}+\eta\cdot F(\Xk{k+1})\leq f^{(k)}$ and $F(\Xk{k+1})\leq f^{(k+1)}$ for any $\eta\in (0,\,1]$.  As a result, it can be concluded that $f^{(k+1)}\leq f^{(k)}$ and $F(\Xk{k})\leq f^{(k)}$. Furthermore, $f^{(k)}$ is a actually convex combination of $F(\Xk{0}),\, \cdots,\, F(\Xk{k})$, and $F(X)$ is bounded below, and thus, $f^{(k)}$ is also bounded below and there exists $F^\infty$ such that $f^{(k)}\rightarrow F^\infty$. Since $f^{(k)}\rightarrow F^\infty$, we obtain $(1-\eta)\cdot f^{(k)}+\eta\cdot F(\Xk{k})\rightarrow F^\infty$ as well, and then  $F(\Xk{k})\rightarrow F^\infty$ as $k\rightarrow\infty$, which completes the proof of \ref{item::agpm_1}. 

\subsubsection{Proof of \ref{item::agpm_2}}
If $\nGamma\succ \nM$, there exists $\epsilon > 0$ such that
$\nGamma\succeq \nM + \epsilon\cdot\I$. From \cref{algorithm::apm*,algorithm::apm}, it can be concluded that 
$$F(\Xk{k+1})\leq f^{(k)}-\delta\cdot\|\Xk{k+1}-\Xk{k}\|^2$$
if $\Xk{k+1}$ is from $\nagm$ or $\nagmo$, or
$$F(\Xk{k+1})\leq f^{(k)}-\epsilon\cdot\|\Xk{k+1}-\Xk{k}\|^2$$
if $\Xk{k+1}$ is from $\gpm$ or $\gpmo$. As a result,  we obtain
\begin{equation}\label{eq::Fk}
F(\Xk{k+1})\leq f^{(k)}-\phi\cdot\|\Xk{k+1}-\Xk{k}\|^2,
\end{equation}
in which $\phi=\min\{\delta,\,\epsilon\}$.
From \cref{eq::Fk} and $f^{(k+1)}=(1-\eta)\cdot f^{(k)}+\eta\cdot F(\Xk{k})$, we obtain
\begin{equation}
\nonumber
f^{(k+1)}\leq f^{(k)}-\eta\cdot\phi\cdot\|\Xk{k+1}-\Xk{k}\|^2.
\end{equation}
Since $f^{(k)}\rightarrow F^\infty$ and $\eta,\,\phi>0$, it can concluded that $\|\Xk{k+1}-\Xk{k}\|\rightarrow 0$ as $k\rightarrow\infty$, which completes the proof of \ref{item::agpm_2}.

\subsubsection{Proof of \ref{item::agpm_3}} For $\agpm$, note that $N_0=1$, then  $\Xk{k+1}$ is actually evaluated as
\begin{equation}\label{eq::minGY}
\Xk{k+1}=\arg\min_{X\in \R^{d\times n}\times SO(d)^n} G(X|\Yk{k})
\end{equation}
if $\Xk{k+1}$ is from $\nagm$, in which 
\begin{equation}\label{eq::Yk}
\Yk{k}=X^{(k)}+\dfrac{s^{(k)}-1}{s^{(k+1)}}\left(X^{(k)}-X^{(k-1)}\right)
\end{equation}
From \cref{eq::prox_F}, we obtain  
\begin{equation}\label{eq::gradGY}
\begin{aligned}
\grad G(\Xk{k+1}|\Yk{k})=&\grad F(\Xk{k+1})+\QQ_{\Xk{k+1}}\big(\nabla F(\Yk{k})-\nabla F(\Xk{k+1})\big)+\\
&\QQ_{\Xk{k+1}}\big((\Xk{k+1}-\Yk{k})\nGamma\big).
\end{aligned}
\end{equation}
Similar to \cref{eq::gradFx,eq::gradFxnorm}, it can be shown that
\begin{multline}
	\nonumber
	\grad F(\Xk{k+1})
	=\QQ_{\Xk{k+1}}\big(\nabla F(\Xk{k+1})-\nabla F(\Yk{k})\big)+\\
	\QQ_{\Xk{k+1}}\big((\Yk{k}-\Xk{k+1})\nGamma\big)
\end{multline}
and
\begin{equation}\label{eq::gradFynorm}
	\|\grad F(\Xk{k+1})\|\leq \|\QQ_{\Xk{k+1}}\|_2\cdot(L+\|\nGamma\|_2)\cdot\|\Xk{k+1}-\Yk{k}\|.
\end{equation}
From \cref{algorithm::nag}, note that $s^{(k)}\geq 1$, and thus
\begin{equation}\label{eq::skk}
	0\leq \frac{s^{(k)}-1}{s^{k+1}}=\frac{2s^{(k)}-2}{\sqrt{4s^{(k)}{}^2+1}+1}\leq \frac{2s^{(k)}-2}{2s^{(k)}}\leq 1.
\end{equation}
As a result of \cref{eq::Yk,eq::skk}, it is straightforward to show that
\begin{multline}\label{eq::sk}
	\|\Xk{k+1}-\Yk{k}\|\leq \|\Xk{k+1}-\Xk{k}\| + \frac{s^{(k)}-1}{s^{(k+1)}}\|\Xk{k}-\Xk{k-1}\|\\
	\leq \|\Xk{k+1}-\Xk{k}\| + \|\Xk{k}-\Xk{k-1}\|,
\end{multline}
from which and \cref{theorem::nopt}\ref{item::agpm_3}, we obtain
\begin{equation}\label{eq::XY}
\|\Xk{k+1}-\Yk{k}\|\rightarrow 0.
\end{equation}
Then, in terms of  $\Xk{k+1}$ resulting from $\nagm$, since $\QQ_{\Xk{k+1}}$, $\|\nGamma\|_2$ and $L$ are nonnegative and bounded, we obtain from \cref{eq::XY,eq::gradFynorm} that $\|\grad F(\Xk{k+1})\|\rightarrow 0$ as $k\rightarrow\infty$ if $\nGamma\succ \nM$. On the other hand, following a similar derivation of \cref{eq::gradFxnorm1} in \cref{theorem::opt}\ref{item::opt_4}, in terms of $\Xk{k+1}$ from $\gpm$, we also obtain $\|\grad F(\Xk{k+1})\|\rightarrow 0$ as $k\rightarrow\infty$ if $\nGamma\succ \nM$. Therefore, no matter $\Xk{k+1}$ is from $\nagm$ or $\gpm$, it can be concluded that $$\|\grad F(\Xk{k+1})\|\rightarrow 0$$
always holds as $k\rightarrow\infty$ if $\nGamma\succ \nM$.

For $\agpmo$, the proof is similar to that of $\agpm$, which is omitted due to space limitation.

\subsubsection{Proof of \ref{item::agpm_4} and \ref{item::agpm_5}} Note that $\nGamma=\alpha\cdot\I +\nOmega \succ \nM$ if $\alpha >0$. Then, the proofs of \ref{item::agpm_4} and \ref{item::agpm_5} are implementation of \ref{item::agpm_2} and \ref{item::agpm_3} of \cref{theorem::nopt}, respectively.

\section{Experiment Results}\label{appendix::results}
\subsection{The Comparisons of Convergence on SLAM Datasets}
\begin{figure}[!htpb]
	\vspace{-2em}
	\setcounter{subfigure}{0}
	\centering
	\begin{tabular}{cc}
		\hspace{5mm}\subfloat[][\tt ais2klinik]{\includegraphics[trim =10mm 6mm 1mm 0mm,width=0.47\textwidth]{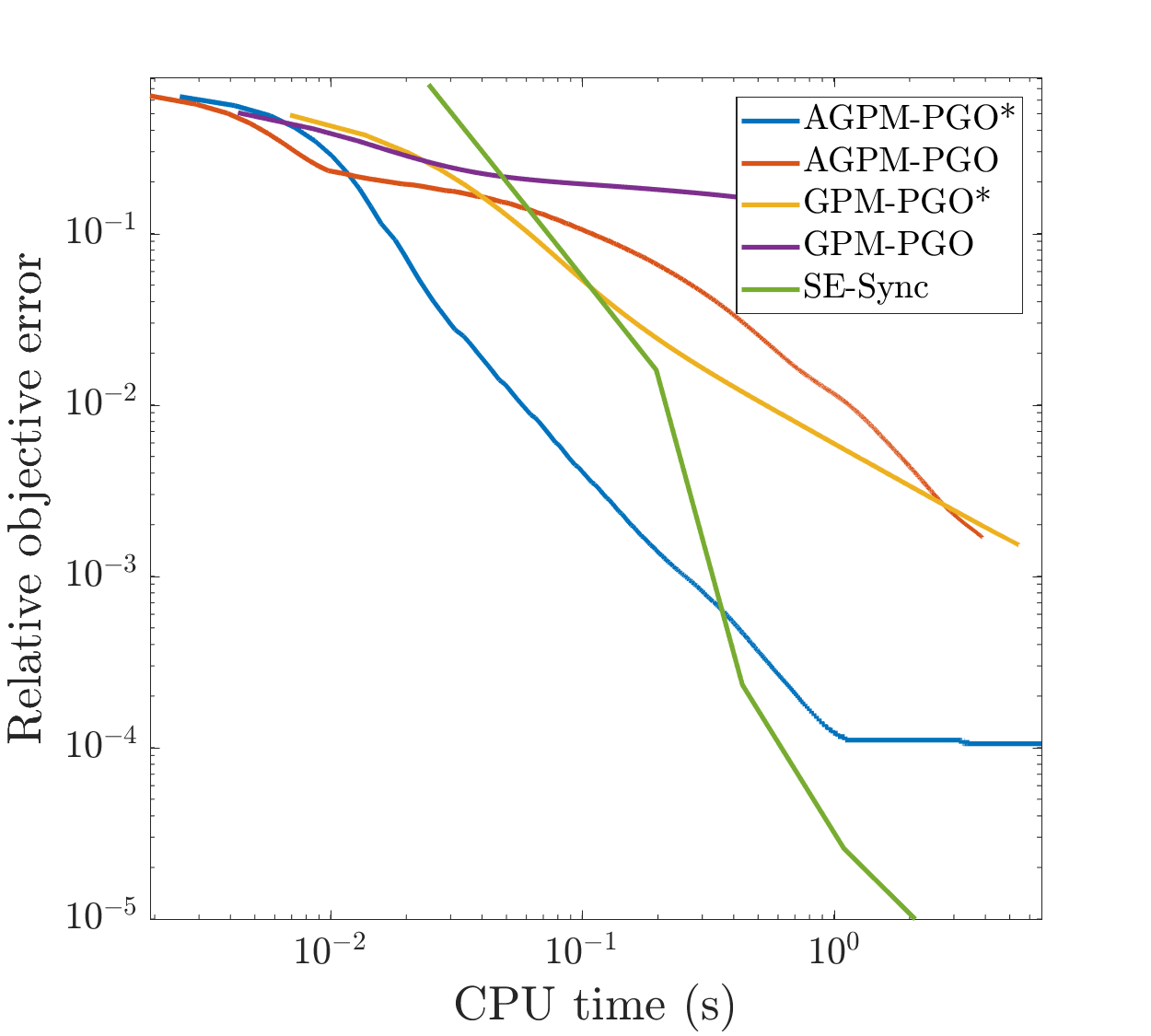}}\hspace{2.5em} &
		\subfloat[][\tt city]{\includegraphics[trim =10mm 6mm 1mm 0mm,width=0.47\textwidth]{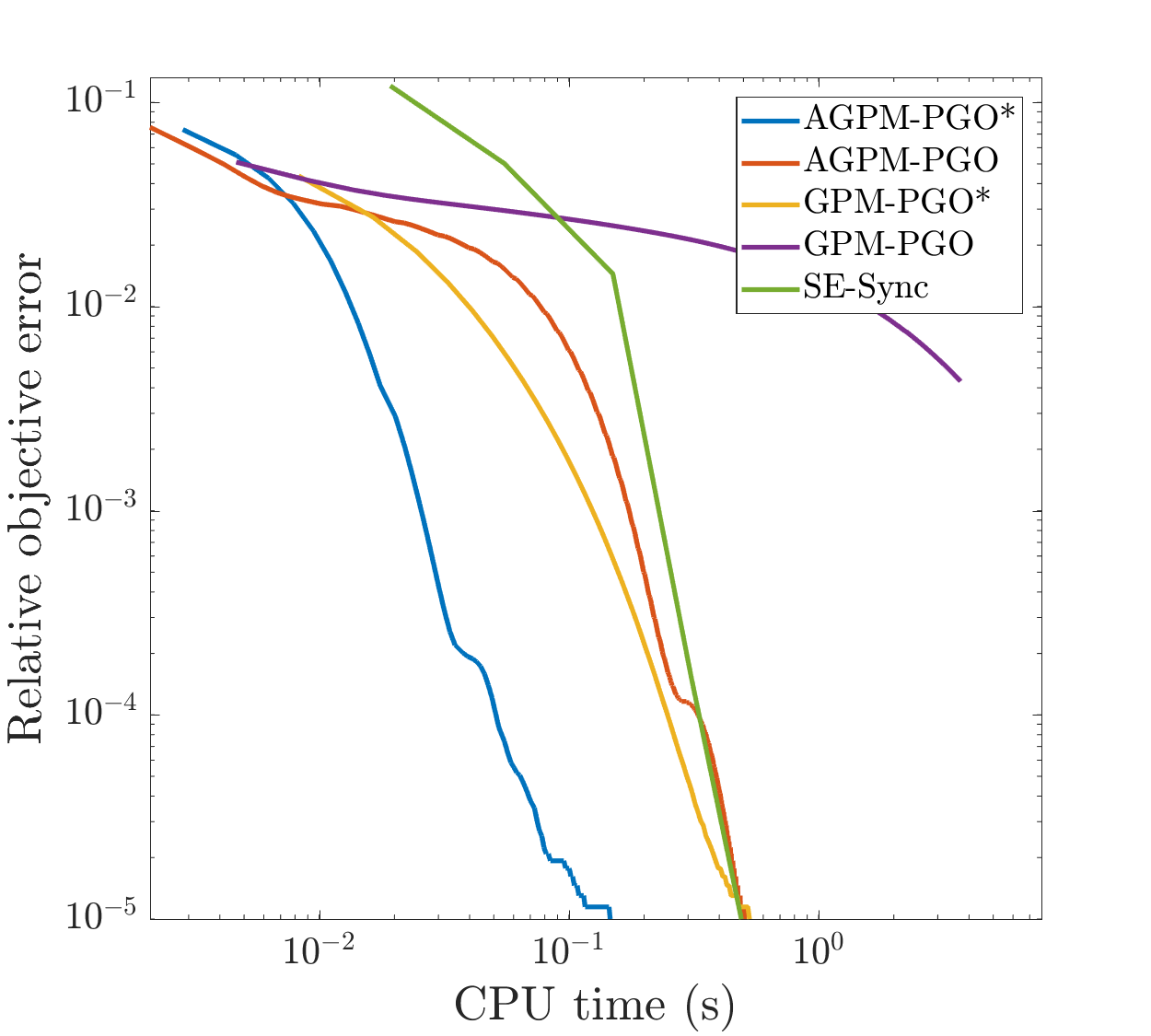}}\\
		\hspace{5mm}\subfloat[][\tt CSAIL]{\includegraphics[trim =10mm 6mm 1mm 0mm,width=0.47\textwidth]{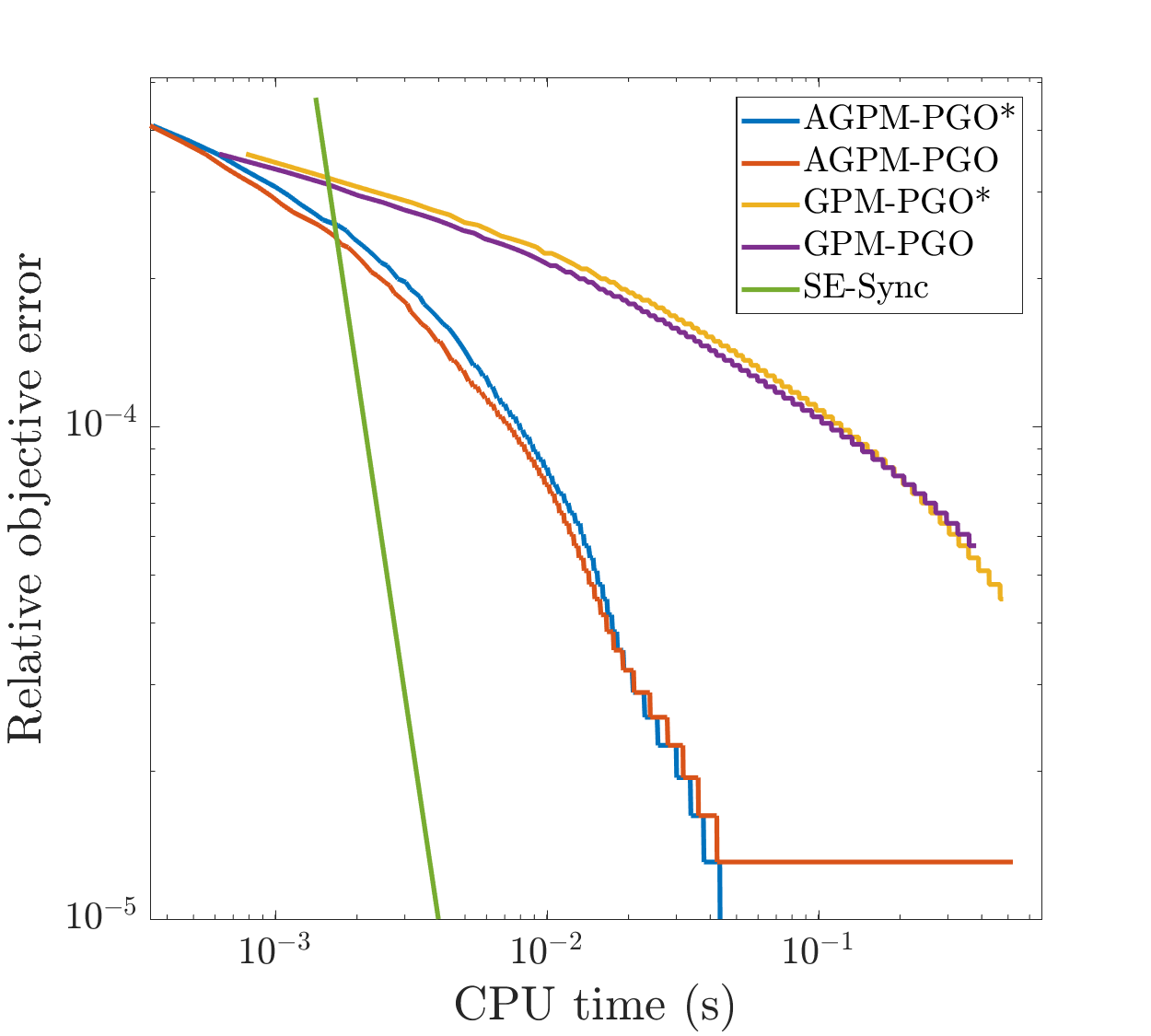}}\hspace{2.5em} &
		\subfloat[][\tt manhattan]{\includegraphics[trim =10mm 6mm 1mm 0mm,width=0.47\textwidth]{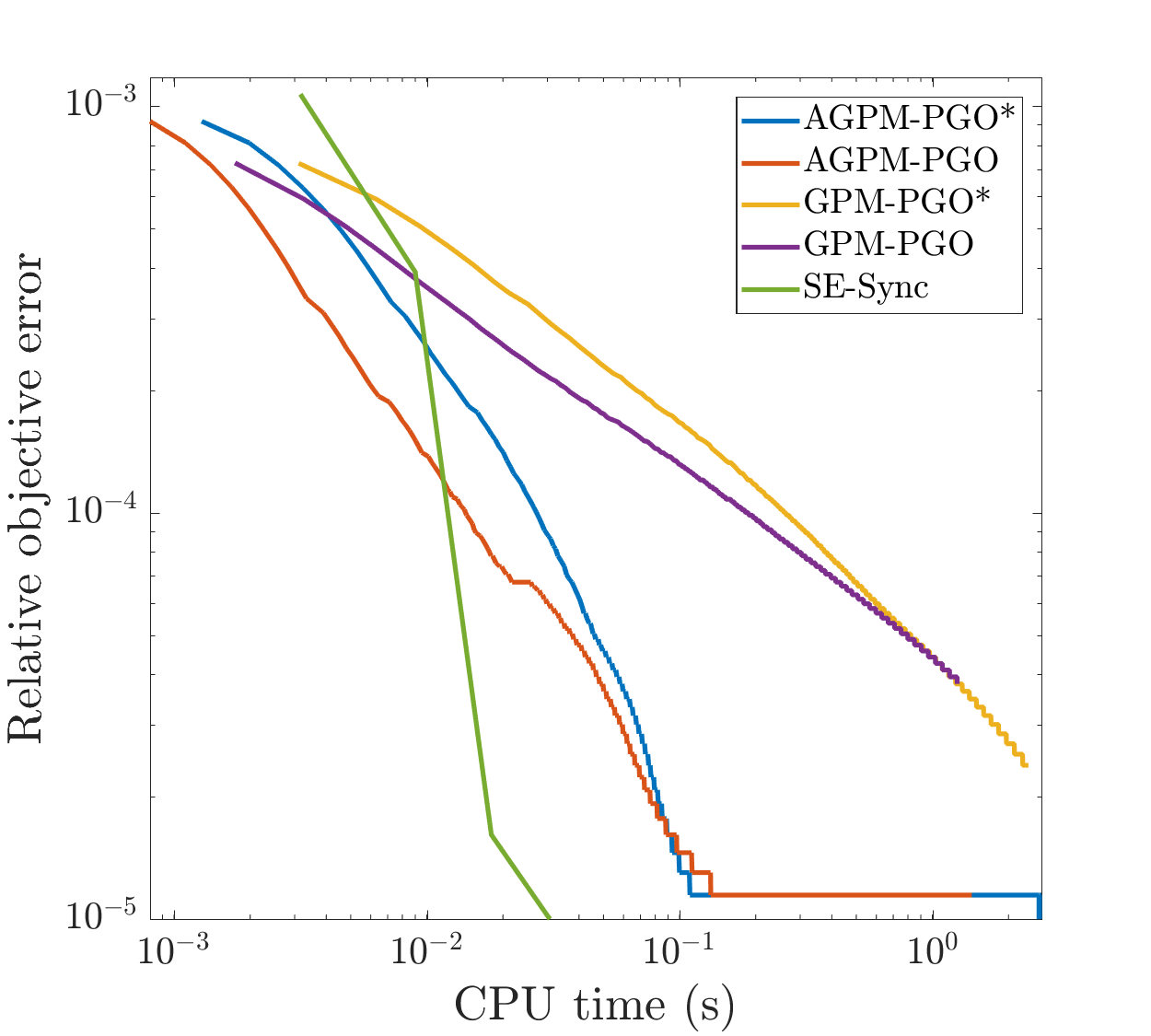}}\\
		\hspace{5mm}\subfloat[][\tt intel]{\includegraphics[trim =10mm 6mm 1mm 0mm,width=0.47\textwidth]{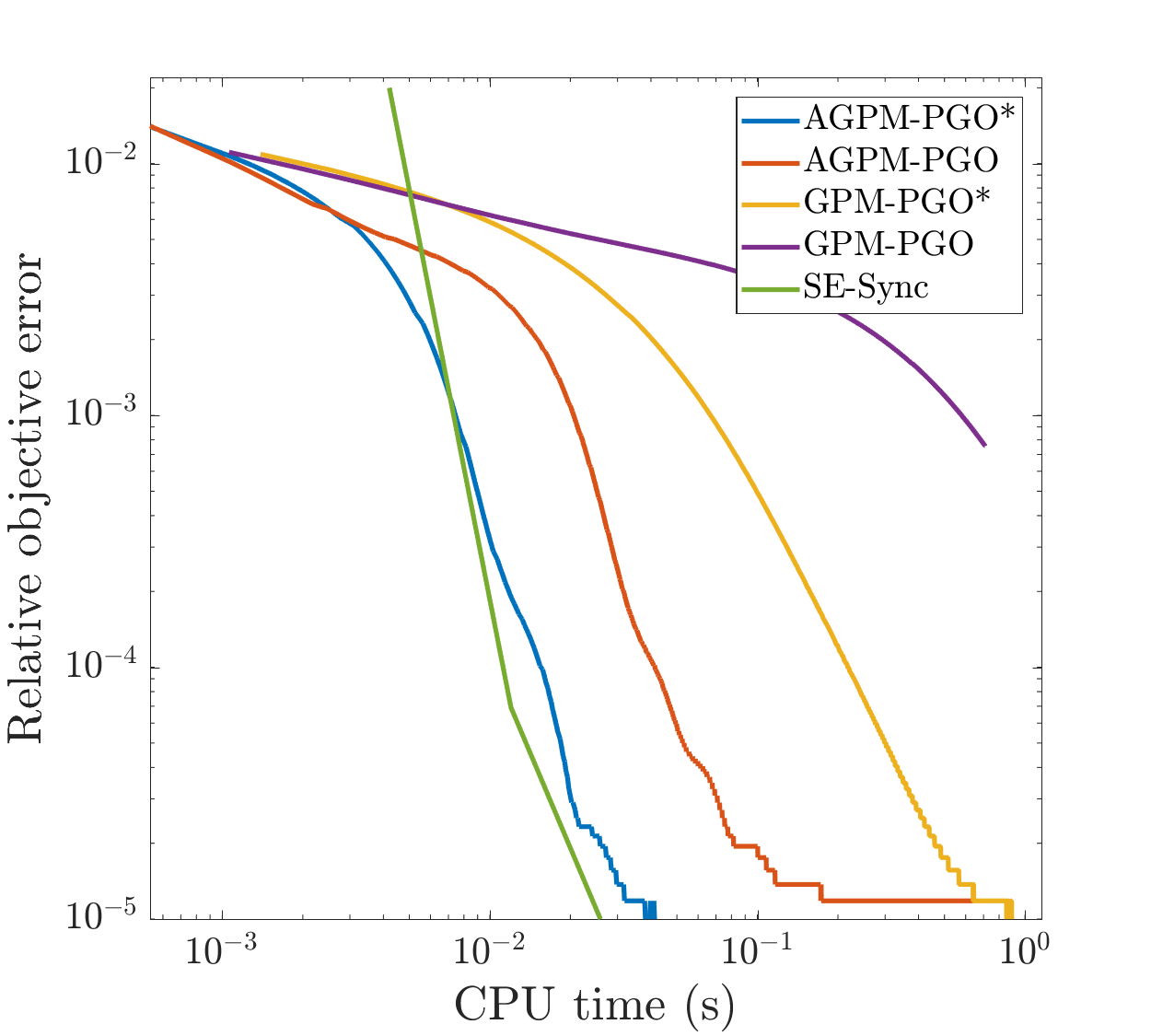}}\hspace{2.5em} &
		\subfloat[][\tt cubicle]{\includegraphics[trim =10mm 6mm 1mm 0mm,width=0.47\textwidth]{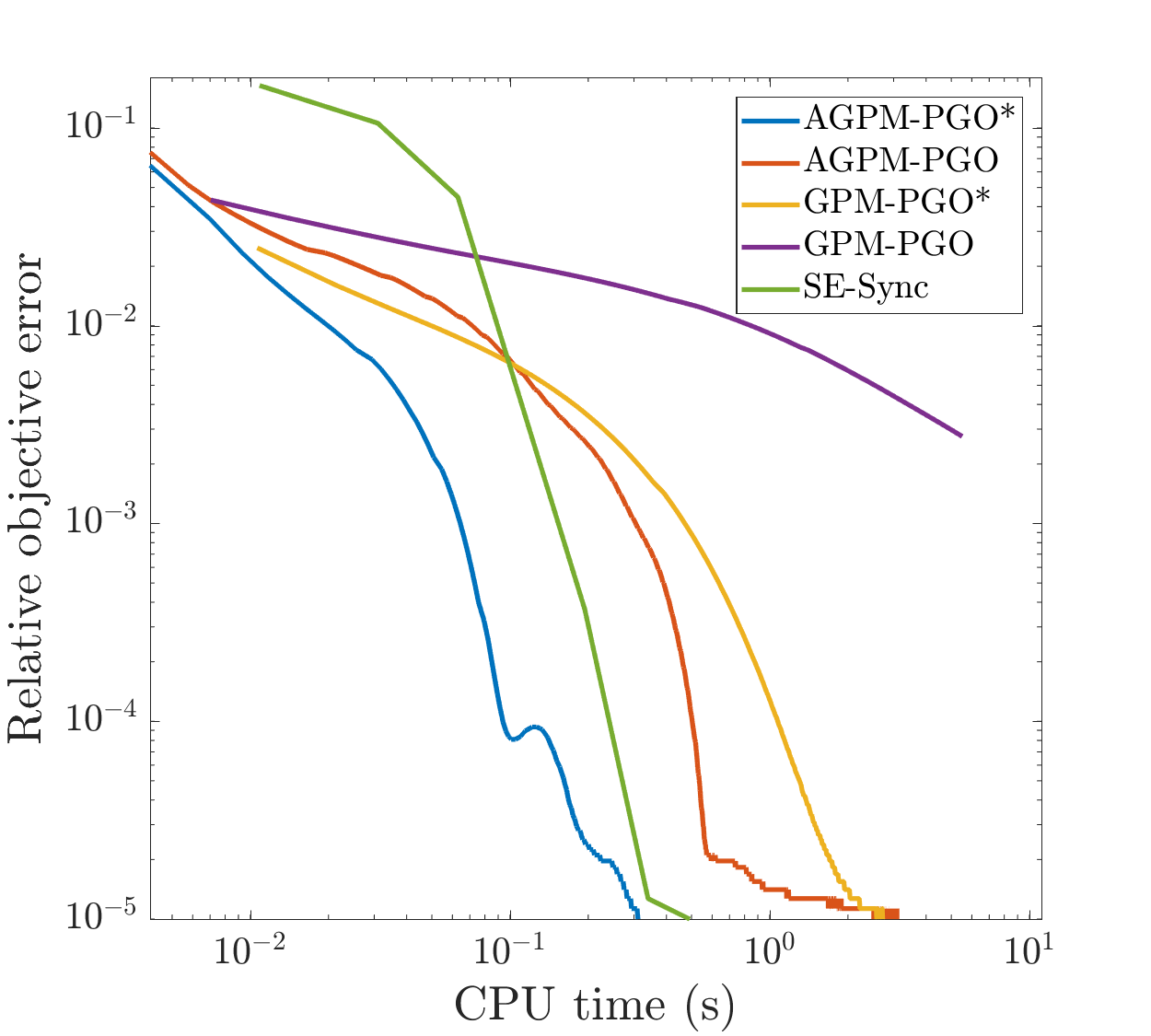}}\\
	\end{tabular}
\end{figure}

\begin{figure}[!htpb]
	\vspace{-2em}
	\setcounter{subfigure}{6}
	\centering
	\begin{tabular}{cc}
		\hspace{5mm}\subfloat[][\tt garage]{\includegraphics[trim =10mm 6mm 1mm 0mm,width=0.47\textwidth]{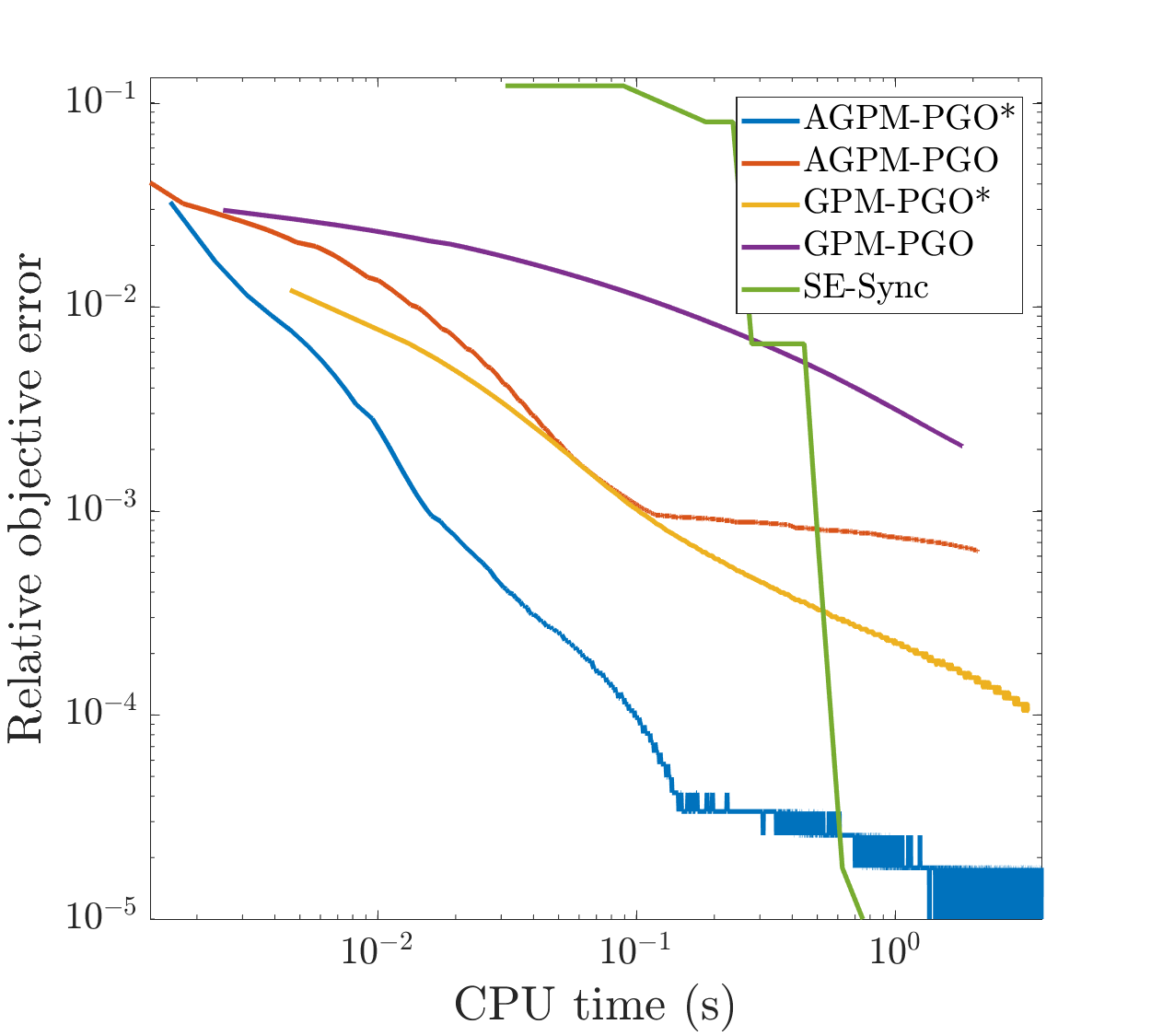}}\hspace{2.5em} &
		\subfloat[][\tt grid]{\includegraphics[trim =10mm 6mm 1mm 0mm,width=0.47\textwidth]{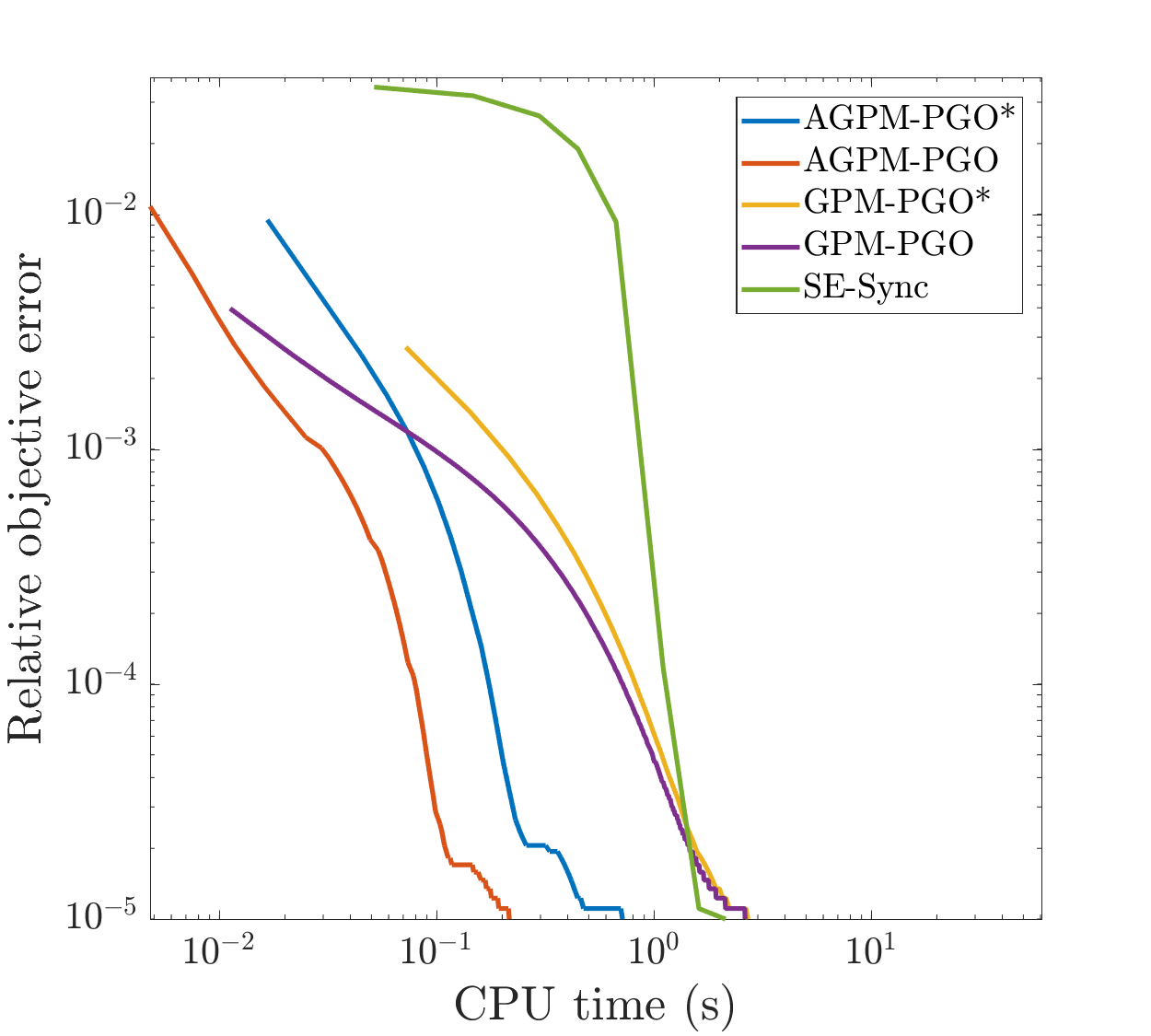}}
	\end{tabular}
\end{figure}

\begin{figure}[!htpb]
	\vspace{-2em}
	\setcounter{subfigure}{8}
	\centering
	\begin{tabular}{cc}
		\hspace{5mm}\subfloat[][\tt rim]{\includegraphics[trim =10mm 6mm 1mm 0mm,width=0.47\textwidth]{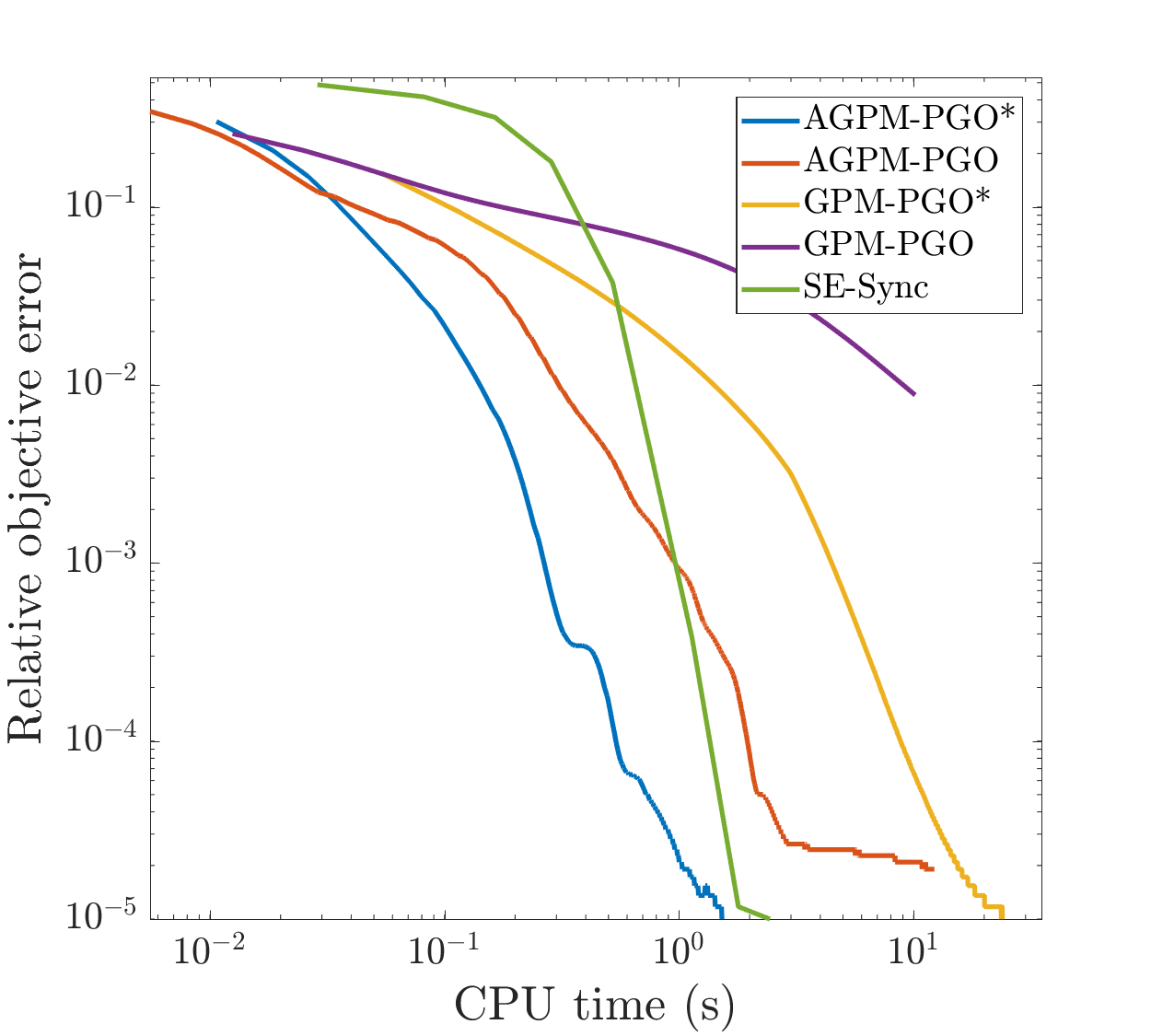}}\hspace{2.5em} &
		\subfloat[][\tt sphere]{\includegraphics[trim =10mm 6mm 1mm 0mm,width=0.47\textwidth]{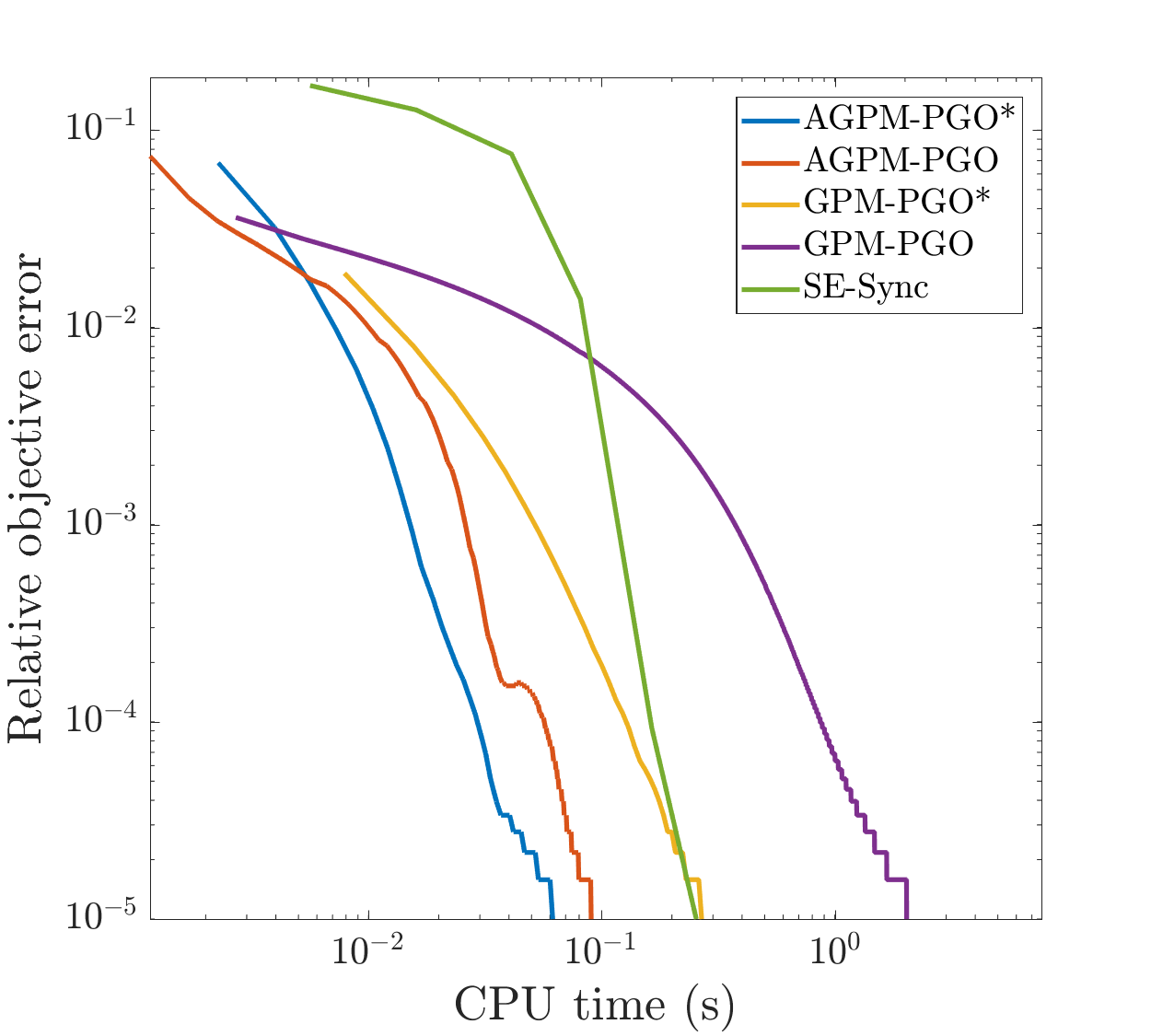}}\\
		\hspace{5mm}\subfloat[][\tt torus]{\includegraphics[trim =10mm 6mm 1mm 0mm,width=0.47\textwidth]{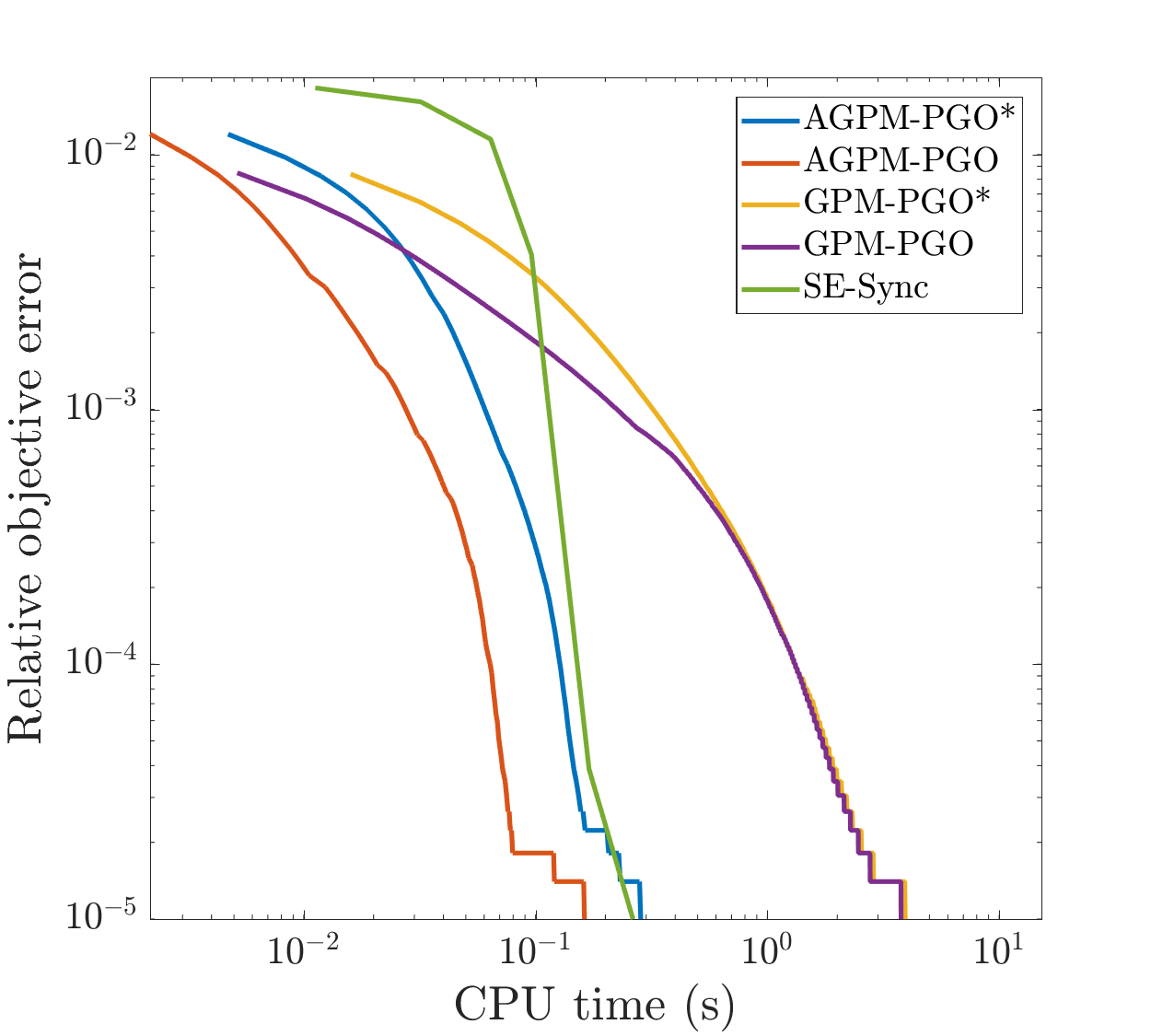}}\hspace{2.5em} &
		\subfloat[][\tt sphere-a]{\includegraphics[trim =10mm 6mm 1mm 0mm,width=0.47\textwidth]{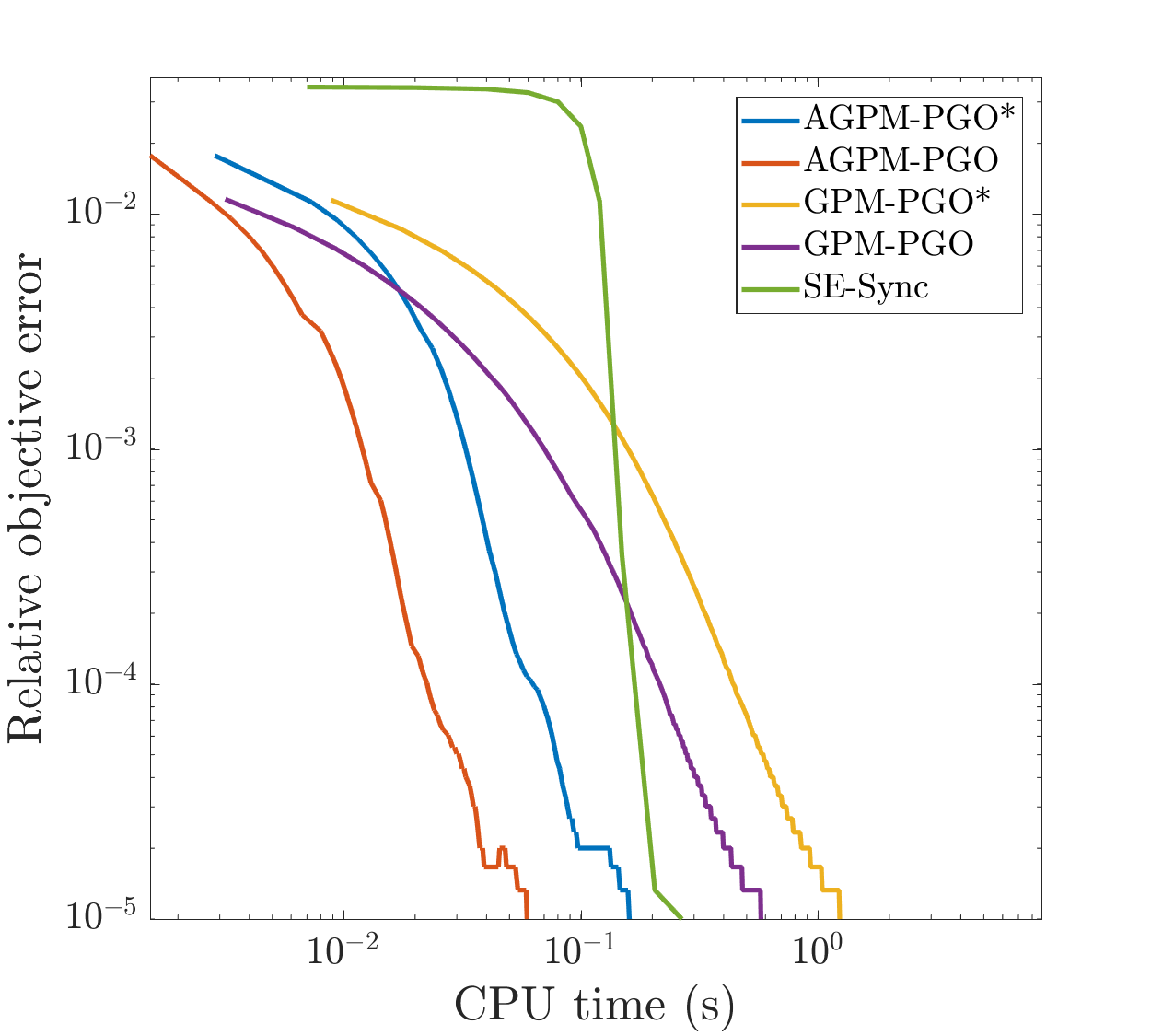}}
	\end{tabular}
\caption{The convergence comparison of $\agpmo$, $\agpm$, $\gpmo$, $\gpm$ and $\sesync$ on 2D and 3D SLAM datasets.}
\end{figure}

\end{document}